\documentclass[smallextended]{svjour3}       
\smartqed  
\usepackage{graphics}
\usepackage{epsfig}
\usepackage{caption}
\usepackage{epstopdf}
\usepackage{mathptmx}
\usepackage{float}
\usepackage{amsmath}
\usepackage{amssymb}
\usepackage{graphicx}
\usepackage{color}
\usepackage{ifpdf}
\usepackage{url}
\usepackage{hyperref}
\usepackage{enumerate}
\usepackage{subfigure}
\graphicspath{{figures/}}
\usepackage{mathptmx, amsmath, amssymb, amsfonts, mathptmx, enumerate, color,enumitem}
\hypersetup{hidelinks} 
\usepackage{algorithm}
\usepackage{algorithmic}
\usepackage[american]{babel}
\usepackage{microtype}
\usepackage[misc]{ifsym} 

\smartqed  

\usepackage{mathptmx}      
\usepackage{latexsym}

\renewcommand{\thetable}{\arabic{section}.\arabic{table}}

\usepackage{booktabs}
\makeatletter
\renewcommand\@biblabel[1]{[#1]}
\makeatother

\begin{document}
\title{Inertial extragradient algorithms for solving variational inequalities  and fixed point problems}
\titlerunning{Inertial extragradient-Mann algorithms}
\author{Bing Tan\and Jingjing Fan\and  Xiaolong Qin}
\authorrunning{B. Tan, J. Fan, X. Qin}

\institute{
Bing Tan \and Jingjing Fan \at  Institute of Fundamental and Frontier Sciences, University of Electronic Science and Technology of China, Chengdu 611731, China  \\
\href{mailto:bingtan72@gmail.com}{bingtan72@gmail.com} (B.Tan), \href{mailto:fanjingjing0324@163.com}{fanjingjing0324@163.com} (J.Fan)
\and
Xiaolong Qin ${(\textrm{\Letter})}$\at
Department of Mathematics, Hangzhou Normal University, Hangzhou 310000, China \\
\href{mailto:qxlxajh@163.com}{qxlxajh@163.com}
}

\date{Received: date / Accepted: date}

\maketitle

\begin{abstract}
The objective of this research is to explore a convex feasibility problem, which consists of a monotone variational inequality problem and a fixed point problem. We introduce four inertial extragradient algorithms that are motivated by the inertial method, the subgradient extragradient method, the Tseng's extragradient method and the Mann-type method endowed with a simple step size. Strong convergence theorems of the algorithms are established under some standard and suitable conditions enforced by the cost operators. Finally, we implement some computational tests to show the efficiency and advantages of the proposed algorithms and compare them with some existing ones.

\end{abstract}

\section{Introduction}
Let $C$ be a nonempty closed convex set in a real Hilbert space $H$  whose induced norm and  inner product are denoted by $\| \cdot\|$ and $\langle \cdot, \cdot \rangle$, respectively. One recalls that the   variational inequality problem (shortly, VIP) is described as follows:
\begin{equation*}\label{VIP}
\begin{aligned}
\text{find } p \in C \text{ such that } \langle Ap,x-p \rangle \geq 0, \quad \forall x \in C \,,
\end{aligned}
\tag{\text{VIP}}
\end{equation*}
where $ A:H\rightarrow H $ is a  nonlinear operator. Let $\mathrm{VI}(C,A)$ represent   the solution set of the problem~\eqref{VIP}. Variational inequality is an essential tool for studying many fields of mathematics and applied science (such as physics, regional, social, engineering, and other issues); see, for example, \cite{IY1,QA,ANQ,SYVS,AIY}. The theories and methods of variational inequalities have been implemented in numerous areas of science and have proven to be successful and creative. The theory has been shown to provide an easy, common, and consistent structure for dealing with possible issues. In the past few decades, researchers have been very interested in developing effective and robust numerical approaches for solving variational inequality problems. In particular, there has been great interest in projection methods and their variants. To see various projection methods, one refers to~\cite{EGM,SEGM,Tseng} and the references therein. It should be mentioned that the extragradient method~\cite{EGM} needs to perform two projection calculations on the feasible set in each iteration, while the subgradient extragradient method~\cite{SEGM} and the Tseng's extragradient method~\cite{Tseng} only require one projection on the feasible set. It is well known that calculating the projection on a non-empty closed convex set is not easy, especially when it has a complex structure. Thus, these two methods greatly improve computational performance in the actual environment.

On the other hand, the fixed point problem is closely related to variational inequalities. A point $p \in H$ is called a fixed point of mapping $T : H \rightarrow H$ if $Tp = p$. We use $\operatorname{Fix}(T)$ to denote the fixed point set of $T$. Our main objective in this paper is to find general solutions to variational inequality problems and fixed point problems. The reason for exploring these problems is that they can be applied to mathematical models, and their constraints can be represented as fixed-point problems and/or variational inequality problems. In recent years, researchers have investigated and proposed many efficient iterative approaches to find common solutions for variational inequality problems and fixed point problems, see, for instance, \cite{Ceng,Nade,CHW,QCC,YP} and the references therein. Recently,  Kraikaew and Saejung~\cite{KS}  proposed an algorithm for finding a common solution to monotone variational inequalities and fixed point problems. This algorithm is based on the Halpern method and the subgradient extragradient method and is now called the Halpern subgradient extragradient method (HSEGM). Indeed, the algorithm is of the form:
\begin{equation*}\label{HSEGM}
\left\{\begin{aligned}
&y^{k}=P_{C}(x^{k}-\gamma A x^{k})\,, \\
&z^{k}=\theta_{k} x^{0}+(1-\theta_{k}) P_{H_{k}}(x^{k}-\gamma A y^{k})\,, \\
&H_{k}=\big\{x \in H \mid \langle x^{k}-\gamma A x^{k}-y^{k}, x-y^{k}\rangle \leq 0\big\}\,, \\
&x^{k+1}=\eta_{k} x^{k}+(1-\eta_{k}) T z^{k}\,,
\end{aligned}\right.
\tag{\text{HSEGM}}
\end{equation*}
where $ P_{C} $ stands for the metric projection of $H$ onto $C$ ($P_{C}x:= \operatorname{argmin}\{\|x-y\|,\, y \in C\}$), mapping $A: H\rightarrow H$ is monotone and $L$-Lipschitz continuous, the step size $ \gamma $ is a fixed number and belongs to $(0,1/L) $, and mapping $T: H \rightarrow H$ is quasi-nonexpansive  (see below for the definition). They proved that the iterative sequence $\{x^{k}\}$ defined in \eqref{HSEGM} converges  to $P_{\operatorname{F i x}(T) \cap \mathrm{VI}(C,A)} x^{0}$ in norm under some suitable conditions. However,   Algorithm~\eqref{HSEGM} needs to know the prior information of the Lipschitz constant of the mapping, which may limit the use of some related algorithms. To overcome such difficulty, a large number of algorithms have been proposed to update the step size through certain adaptive criteria, see, for example, \cite{ShefuNA,Liu,VTEGM}. Recently, Tong and Tian~\cite{VTEGM} proposed a new self-adaptive iterative algorithm to solve variational inequality problems and fixed point problems in a Hilbert space. Their algorithm is motivated by the Tseng's extragradient method, the hybrid steepest descent method and the Mann-type method. The adaptive criterion adopted can guarantee that the algorithm works without knowing the Lipschitz constant of the cost mapping. Their algorithm is described as follows:
\begin{equation*}\label{STEGM}
\left\{\begin{aligned}
&y^{k}=P_{C}(x^{k}-\gamma_{k} A x^{k})\,, \\
&z^{k}=y^{k}-\gamma_{k}(A y^{k}-A x^{k})\,, \\
&t^{k}=(1-\eta_{k}) z^{k}+\eta_{k} T z^{k}\,, \\
&x^{k+1}=(1-\lambda \theta_{k} F) t^{k}\,,
\end{aligned}\right.
\tag{\text{STEGM}}
\end{equation*}
where mapping $A: H\rightarrow H$ is  monotone and $L$-Lipschitz continuous, mapping $ T: H\rightarrow H $ is quasi-nonexpansive with a demiclosedness property and mapping $F: H \rightarrow H$ is strongly monotone and Lipschitz continuous.
The step size $\gamma_{k}$ will be automatically updated in each iteration by selecting the maximum $\gamma \in\left\{\rho, \rho l, \rho l^{2}, \ldots\right\}$ that satisfies $ \gamma\|A x^{k}-A y^{k}\| \leq \phi\|x^{k}-y^{k}\| $ (this rule is called the Armijo-like line search criterion). Under some suitable conditions, the iterative sequence generated by \eqref{STEGM}  converges to $z=P_{\operatorname{Fix}(T)\cap\mathrm{VI}(C,A)}(I-\gamma F)z$ in norm.

In this paper, we focus on the situation that  $T$ is a demicontractive mapping, which covers quasi-nonexpansive mappings. In 2018, Thong and Hieu~\cite{THSEGM} proposed two Mann-type subgradient extragradient algorithms to find common elements of variational inequalities and fixed point problems involving a demicontractive mapping. More precisely, their iterative algorithms are as follows:
\begin{equation*}\label{MSEGM}
\left\{\begin{aligned}
&y^{k}=P_{C}(x^{k}-\gamma A x^{k})\,, \\
&z^{k}=P_{H_{k}}(x^{k}-\gamma A y^{k})\,, \\
&H_{k}=\big\{x \in H \mid \langle x^{k}-\gamma A x^{k}-y^{k}, x-y^{k}\rangle \leq 0\big\}\,, \\
&x^{k+1}=(1-\theta_{k}-\eta_{k}) z^{k}+\eta_{k} T z^{k}\,,
\end{aligned}\right.
\tag{\text{MSEGM}}
\end{equation*}
and
\begin{equation*}\label{MMSEGM}
\left\{\begin{aligned}
&y^{k}=P_{C}(x^{k}-\gamma A x^{k})\,, \\
&z^{k}=P_{H_{k}}(x^{k}-\gamma A y^{k})\,, \\
&H_{k}=\big\{x \in H \mid \langle x^{k}-\gamma A x^{k}-y^{k}, x-y^{k}\rangle \leq 0\big\}\,, \\
&x^{k+1}=(1-\eta_{k})(\theta_{k} z^{k})+\eta_{k} T z^{k}\,,
\end{aligned}\right.
\tag{\text{MMSEGM}}
\end{equation*}
where mapping $A: H\rightarrow H$ is  monotone and $L$-Lipschitz continuous, step size $ \gamma\in(0,1/L) $ and mapping $T: H \rightarrow H$ is  $\lambda$-demicontractive with $0 \leq \lambda<1$. They obtained strong convergence theorems of the suggested algorithms in real Hilbert spaces under some suitable and mild assumptions.

Note that algorithms \eqref{MSEGM} and \eqref{MMSEGM} require to know the prior information of the Lipschitz constant of the cost mapping. In addition, we point out that the method of updating the step size through the Armijo-like criterion may be computationally expensive because it needs to calculate the value of $A$ many times in each iteration. To overcome these shortcomings, one method is to update the step size through some simple calculations in each iteration. Recently, Thong and Hieu~\cite{TVNA} introduced two extragradient viscosity-type iterative algorithms with new simple step size to solve variational inequalities and fixed point problems. Their algorithms are of the following forms:
\begin{equation*}\label{VSEGM}
\left\{\begin{aligned}
&y^{k}=P_{C}(x^{k}-\gamma_{k} A x^{k})\,, \\
&z^{k}=P_{H_{k}}(x^{k}-\gamma_{k} A y^{k})\,, \\
&H_{k}=\big\{x \in H \mid \langle x^{k}-\gamma_{k} A x^{k}-y^{k}, x-y^{k}\rangle \leq 0\big\}\,, \\
&x^{k+1}=\theta_{k} f(x^{k})+(1-\theta_{k})\big[(1-\eta_{k}) z^{k}+\eta_{k} T z^{k}\big]\,,
\end{aligned}\right.
\tag{\text{VSEGM}}
\end{equation*}
and
\begin{equation*}\label{VTEGM}
\left\{\begin{aligned}
&y^{k}=P_{C}(x^{k}-\gamma_{k} A x^{k})\,, \\
&z^{k}=y^{k}-\gamma_{k}(A y^{k}-A x^{k})\,, \\
&x^{k+1}=\theta_{k} f(x^{k})+(1-\theta_{k})\big[(1-\eta_{k}) z^{k}+\eta_{k} T z^{k}\big]\,,
\end{aligned}\right.
\tag{\text{VTEGM}}
\end{equation*}
where mapping $A: H\rightarrow H$ is  monotone and $L$-Lipschitz continuous, mapping $T: H \rightarrow H$ is  $\lambda$-demicontractive and mapping $ f $ is contractive, algorithms \eqref{VSEGM} and \eqref{VTEGM} update the step size $ \{\gamma_{k}\} $ by the following rules:
\[
\gamma_{k+1}=\left\{\begin{array}{ll}
\min \left\{\frac{\phi\|x^{k}-y^{k}\|}{\|A x^{k}-A y^{k}\|}, \gamma_{k}\right\}, & \text { if } A x^{k}-A y^{k} \neq 0; \\
\gamma_{k}, & \text { otherwise}.
\end{array}\right.
\]
It should be highlighted that algorithms \eqref{STEGM}, \eqref{VSEGM} and \eqref{VTEGM} only need to compute the projection on the feasible set $ C $ once in each iteration, and they can work without the prior information of the Lipschitz constant of cost mapping. These algorithms have achieved strong convergence theorems in real Hilbert spaces Under some suitable conditions.

In recent years, the development of fast iterative algorithms has aroused great interest from scientific researchers. The inertial algorithm is a two-stage iterative procedure. Its main feature is to use the previous two iterations to represent the next iteration. Many authors have used inertial methods to build a large number of iterative algorithms that can improve the convergence speed; see, for example,  \cite{Fan,iFB,tanjnca,Tan,SLD} and the references therein. These inertial-type algorithms have better numerical performance than algorithms without inertial terms.

Motivated and stimulated by results as mentioned above, in this paper, we suggest four new inertial Mann-type extragradient algorithms by inserting the inertial terms into the Tseng's extragradient algorithm and the subgradient extragradient algorithm. They are used to find a common element of the solution set of the monotone variational inequality problem and the fixed point set of a demicontractive mapping. We automatically update the step size in each iteration through a simple adaptive criterion, which allows the algorithms to work without knowing the Lipschitz constant of the mapping in advance. We obtain strong convergence  of these algorithms under some standard and mild hypotheses. Finally, we give several numerical examples to support the theoretical results. Numerical results show that the new algorithm converges faster than existing algorithms.

The remaining part of the paper proceeds as follows: In the next Section, we recall some preliminary results. In Section~\ref{sec3}, we analyze the convergence of the proposed algorithm. In Section~\ref{sec4}, some computational tests are provided to illustrate the numerical behavior of the proposed algorithms and compare them with existing  ones. Finally, a brief summary is given in Section~\ref{sec5}, the last section.
\section{Preliminaries}\label{sec2}
Throughout this paper, we always assume that $H$ represents Hilbert space and $C$ denotes the nonempty convex and closed subset of $ H $. The weak convergence and strong convergence of $\big\{x^{k}\big\}_{k=1}^{\infty}$ to $x$ are represented by $x^{k} \rightharpoonup x$ and $x^{k} \rightarrow x$, respectively. For each $x, y \in H$ and $\theta \in \mathbb{R}$, we have the following basic inequalities:
\begin{itemize}
\item $\|x+y\|^{2} \leq\|x\|^{2}+2\langle y, x+y\rangle$;
\item $\|\theta x+(1-\theta) y\|^{2}=\theta\|x\|^{2}+(1-\theta)\|y\|^{2}-\theta(1-\theta)\|x-y\|^{2}$.
\end{itemize}

It is known that $P_{C}$ has the following basic properties:
\begin{itemize}
\item $ \langle x-P_{C} x, y-P_{C} x\rangle \leq 0, \, \forall y \in C $;
\item $\|P_{C} x-P_{C} y\|^{2} \leq\langle P_{C} x-P_{C} y, x-y\rangle, \,\forall y \in H$.
\end{itemize}

\begin{definition}
Suppose that nonlinear operator $T: H \rightarrow H$ satisfies  $\operatorname{Fix}(T) \neq \emptyset $. If for any $\left\{x^{k}\right\}\subset H$, $ x^{k} \rightharpoonup x \text { and }(I-T) x^{k} \rightarrow 0 $ implies that $ x \in \operatorname{Fix}(T) $. Then $I-T$ is said to be demiclosed at zero.
\end{definition}

\begin{definition}
For any $ x, y \in H, z \in \operatorname{Fix}(\mathcal{A})$, operator $\mathcal{A}: H \rightarrow H$ is said to be:   
\begin{itemize}
\item \emph{$L$-Lipschitz continuous} with $L>0$ if
\[
\|\mathcal{A} x-\mathcal{A} y\| \leq L\|x-y\|\,.
\]
\item \emph{monotone} if
\[
\langle \mathcal{A} x-\mathcal{A} y, x-y\rangle \geq 0 \,.
\]
\item \emph{quasi-nonexpansive} if
\[
\|\mathcal{A} x-z\| \leq\|x-z\|\,.
\]
\item \emph{$\lambda$-strictly pseudocontractive} with $0 \leq \lambda<1$ if
\[
\|\mathcal{A} x-\mathcal{A} y\|^{2} \leq\|x-y\|^{2}+\lambda\|(I-\mathcal{A}) x-(I-\mathcal{A}) y\|^{2}\,.
\]
\item \emph{$\eta$-demicontractive} with $0 \leq \eta<1$ if
\begin{equation}\label{eq21}
\|\mathcal{A} x-z\|^{2} \leq\|x-z\|^{2}+\eta\|(I-\mathcal{A}) x\|^{2}\,,
\end{equation}
or equivalently
\begin{equation}\label{eq22}
\langle \mathcal{A} x-x, x-z\rangle \leq \frac{\eta-1}{2}\|x-\mathcal{A} x\|^{2}\,,
\end{equation}
or equivalently
\begin{equation}\label{eq23}
\langle \mathcal{A} x-z, x-z\rangle \leq\|x-z\|^{2}+\frac{\eta-1}{2}\|x-\mathcal{A} x\|^{2}\,.
\end{equation}
\end{itemize}
\end{definition}

\begin{remark}
According to the above definitions, we can easily see the following facts:
\begin{itemize}
\item Every strictly pseudocontractive mapping with a nonempty fixed point set is demicontractive.
\item The type of demicontractive mappings includes the type of quasi-nonexpansive mappings.
\end{itemize}
\end{remark}

The following lemmas are crucial in the proof of convergence of the algorithms.
\begin{lemma}[\cite{KS}]\label{lem21}
Suppose that $A : H \rightarrow H$ is a monotone and $ L $-Lipschitz continuous mapping. Let $T=P_{C}(I-\phi A)$, where $\phi>0$. If $\{x^{k}\}\subset H$ satisfying $x^{k} \rightharpoonup u$ and $x^{k}-T x^{k} \rightarrow 0$, then $u \in \mathrm{VI}(C, A)=\operatorname{Fix}(T)$.
\end{lemma}
\begin{lemma}[\cite{Mainge2}]\label{lem22}
Suppose that $\{b^{k}\}$ is a nonnegative sequence. If there exists a subsequence $\{b^{k_{j}}\}$ of $\left\{b^{k}\right\}$ satisfies $b^{k_{j}}<b^{k_{j}+1},\forall j \in \mathbb{N}$. Then, there exists a nondecreasing sequence $\left\{m_{k}\right\}$ of $\mathbb{N}$ satisfies $\lim _{k \rightarrow \infty} m_{k}=\infty$. Moreover, for all (sufficiently large) $k \in \mathbb{N}$, the following inequalities are satisfied:
$b^{m_{k}} \leq b^{m_{k}+1}$ and $b^{k} \leq b^{m_{k}+1}$. Actually, $m_{k}$ is the largest number $n$ in the set $\{1,2, \ldots, k\}$ satisfies $b^{n}<b^{n+1}$.
\end{lemma}
\begin{lemma}[\cite{xu}]\label{lem23}
Suppose that $\{a^{k}\}$ is a  nonnegative sequence  satisfying
$a^{k+1} \leq \theta_{k} b^{k}+(1-\theta_{k}) a^{k},\forall k>0$, where $\left\{\theta_{k}\right\} \subset(0,1)$ and $\{b^{k}\}$ is a sequence such that $\sum_{k=0}^{\infty} \theta_{k}=\infty$ and $\limsup_{k \rightarrow \infty} b^{k} \leq 0$. Then, $\lim _{k \rightarrow \infty} a^{k}=0$.
\end{lemma}
\begin{lemma}\label{lem24}
Suppose that $T : H \rightarrow H$ is $ \eta $-demicontractive with $\operatorname{Fix}(T) \neq \emptyset$. Set $T_{\lambda}=\lambda T+(1-\lambda) I$, where I stands for identity mapping and $ \lambda \in(0,1-\eta)$. Then:
\begin{enumerate}[label=(\roman*)]
\item $\operatorname{Fix}(T)=\operatorname{Fix}(T_{\lambda})$;
\item $\|T_{\lambda} x-u\|^{2} \leq\|x-u\|^{2}-\frac{1}{\lambda}(1-\eta-\lambda)\|(I-T_{\lambda}) x\|^{2},\quad  \forall u \in \operatorname{Fix}(T), x \in H$;
\item $\operatorname{Fix}(T)$ is a convex and closed set.
\end{enumerate}
\end{lemma}
\begin{proof}
\begin{enumerate}[label=(\roman*)]
\item It is obvious.

\item According to the definition of $ T_{\lambda} $ and \eqref{eq22}, we obtain
\[
\begin{aligned}
\|T_{\lambda} x-u\|^{2}  
&=\|x-u\|^{2}+2 \lambda\langle x-u, T x-x\rangle+\lambda^{2}\|T x-x\|^{2} \\
& \leq\|x-u\|^{2}+\lambda(\eta-1)\|T x-x\|^{2}+\lambda^{2}\|T x-x\|^{2} \\
&=\|x-u\|^{2}-\frac{1}{\lambda}(1-\eta-\lambda)\|(I-T_{\lambda}) x\|^{2}\,.
\end{aligned}
\]
\item It is a consequence of \cite[Proposition 1]{YO}.
\end{enumerate}
\end{proof}
\section{Main results}\label{sec3}
In this section, we present four inertial extragradient approaches to solve variational inequality problems and fixed point problems, and analyze their convergence. These algorithms are inspired and driven by the subgradient extragradient method, the Tseng's extragradient method and the Mann-type method. In particular, we have added inertial term and new step size, which makes these algorithms have faster convergence speed and do not need to know the prior information of Lipschitz constant in advance. First, we assume that our proposed Algorithm~\ref{alg1} and Algorithm~\ref{alg2} satisfy the subsequent four assumptions.
\begin{enumerate}[label=(C\arabic*)]
\item The mapping $A: H \rightarrow H$ is monotone and $L$-Lipschitz continuous on $H$. \label{con1}
\item The mapping $T: H \rightarrow H$ is  $\lambda$-demicontractive such that $(I-T)$ is demiclosed at zero. \label{con2}
\item The solution set $\operatorname{Fix}(T) \cap \mathrm{VI}(C, A) \neq \emptyset $. \label{con3}
\item  $ \{\zeta_{k}\} $ is a sequence of positive numbers and satisfies $\lim_{k \rightarrow \infty} \frac{\zeta_{k}}{\theta_{k}}=0$, where $ \{\theta_{k}\} $ is a sequence of $ (0,1) $, and satisfies $\sum_{n=1}^{\infty} \theta_{k}=\infty$ and $\lim _{k \rightarrow \infty} \theta_{k}=0$. Let the positive sequence $\left\{\eta_{k}\right\}$ satisfy $\eta_{k} \in(a, b) \subset(0,(1-\lambda)(1-\theta_{k}))$, for some $a>0$, $b>0$. \label{con4}
\end{enumerate}

\subsection{The inertial Mann-type subgradient extragradient algorithm}
Now, we present an inertial Mann-type subgradient extragradient algorithm to solve variational inequality problems and fixed point problems. The details of the algorithm are described as follows:
\begin{algorithm}[h]
\caption{The inertial Mann-type subgradient extragradient algorithm}
\label{alg1}
\begin{algorithmic}
\STATE {\textbf{Initialization:} Give $ \delta>0 $, $\gamma_{1}>0$, $\phi \in(0,1)$. Let $x^{0},x^{1} \in H$ be two arbitrary initial points.}
\STATE \textbf{Iterative Steps}: Calculate the next iteration point $ x^{k+1} $ as follows:
\STATE \textbf{Step 1.} Given two previously known iteration points $x^{k-1}$ and $x^{k}(k \geq 1) $. Calculate
\[s^{k}=x^{k}+\delta_{k}(x^{k}-x^{k-1})\,,\]
where
\begin{equation}\label{alpha}
\delta_{k}=\left\{\begin{array}{ll}
\min \bigg\{\dfrac{\zeta_{k}}{\|x^{k}-x^{k-1}\|}, \delta\bigg\}, & \text { if } x^{k} \neq x^{k-1}; \\
\delta, & \text { otherwise}.
\end{array}\right.
\end{equation}
\STATE \textbf{Step 2.} Calculate  
\[y^{k}=P_{C}(s^{k}-\gamma_{k} A s^{k})\,,\]
\STATE \textbf{Step 3.} Calculate 
\[z^{k}=P_{H_{k}}(s^{k}-\gamma_{k} A y^{k})\,,\]
where $ H_{k}:=\left\{x \in H \mid\langle s^{k}-\gamma_{k} A s^{k}-y^{k}, x-y^{k}\rangle \leq 0\right\} $.
\STATE \textbf{Step 4.} Calculate
\[x^{k+1}=(1-\theta_{k}-\eta_{k}) z^{k}+\eta_{k} T z^{k}\,,\]
and update
\begin{equation}\label{lambda2}
\gamma_{k+1}=\left\{\begin{array}{ll}
\min \left\{\dfrac{\phi\|s^{k}-y^{k}\|}{\|A s^{k}-A y^{k}\|}, \gamma_{k}\right\}, & \text { if } A s^{k}-A y^{k} \neq 0; \\
\gamma_{k}, & \text { otherwise}.
\end{array}\right.
\end{equation}
\end{algorithmic}
\end{algorithm}
\begin{remark}\label{rem31}
It follows from \eqref{alpha} that
\[
\lim _{k \rightarrow \infty} \frac{\delta_{k}}{\theta_{k}}\|x^{k}-x^{k-1}\|=0\,.
\]
Indeed, for all $k \ge 0$, we get that $\delta_{k}\|x^{k}-x^{k-1}\| \leq \zeta_{k}$, which, together with $\lim _{k \rightarrow \infty} \frac{\zeta_{k}}{\theta_{k}}=~0$ implies that
\[
\lim _{k \rightarrow \infty} \frac{\delta_{k}}{\theta_{k}}\|x^{k}-x^{k-1}\| \leq \lim _{k \rightarrow \infty} \frac{\zeta_{k}}{\theta_{k}}=0\,.
\]
\end{remark}
The following two lemmas are very important for the convergence analysis of the algorithms.
\begin{lemma}\label{lem31}
The sequence $\left\{\gamma_{k}\right\}$ defined in \eqref{lambda2} is a nonincreasing and satisfies
\[
\lim _{k \rightarrow \infty} \gamma_{k}=\gamma \geq \min \Big\{\gamma_{1}, \frac{\phi}{L}\Big\}\,.
\]
\end{lemma}
\begin{proof}
It follows from \eqref{lambda2} that $\gamma_{k+1} \leq \gamma_{k}$ for all $n \in \mathbb{N} $. Hence, $\left\{\gamma_{k}\right\}$ is nonincreasing. On the other hand, we get that $\|A s^{k}-A y^{k}\| \leq L\|s^{k}-y^{k}\|$ since $A$ is $L$-Lipschitz continuous. Consequently
\[
\phi \frac{\|s^{k}-y^{k}\|}{\|A s^{k}-A y^{k}\|} \geq \frac{\phi}{L} \,\,,\text {  if  }\,\, A s^{k} \neq A y^{k}\,,
\]
which together with \eqref{lambda2} implies that $ \gamma_{k} \geq \min \{\gamma_{1}, \frac{\phi}{L}\} $. Therefore, from the sequence $ \{\gamma_{k}\} $  is nonincreasing and lower bounded, we have $\lim _{k \rightarrow \infty} \gamma_{k}=\gamma \geq \min \big\{\gamma_{1}, \frac{\phi}{L}\big\}$.\qed
\end{proof}

\begin{lemma}[\cite{tanarxiv}]\label{lem32}
Suppose that Conditions \ref{con1} and \ref{con3} hold. Let sequence $\{z^{k}\}$ be generated by Algorithm~\ref{alg1}. Then, for all $u \in \mathrm{VI}(C, A)$, we have, 
\begin{equation}\label{q}
\|z^{k}-u\|^{2} \leq\|s^{k}-u\|^{2}-\Big(1-\phi \frac{\gamma_{k}}{\gamma_{k+1}}\Big)\|y^{k}-s^{k}\|^{2}-\Big(1-\phi \frac{\gamma_{k}}{\gamma_{k+1}}\Big)\|z^{k}-y^{k}\|^{2}\,.
\end{equation}
\end{lemma}

\begin{theorem}\label{thm31}
Suppose that Conditions \ref{con1}--\ref{con4} hold. Then the iterative sequence $\{x^{k}\}$ generated by Algorithm~\ref{alg1} converges to $u \in \operatorname { Fix }(T) \cap \mathrm{VI}(C, A)$  in norm, where $\|u\|=\min \{\|p\|: p \in \operatorname { Fix }(T) \cap \mathrm{VI}(C, A)\}$.
\end{theorem}
\begin{proof}
It follows from Lemma~\ref{lem24} that  $ \operatorname{Fix}(T) $ is a  convex and closed set. Note that $\mathrm{VI}(C, A)$ is also a closed and convex set. According to the definition of $ u $, we have $u=P_{\mathrm{VI}(C, A) \cap \operatorname { Fix }(T)} 0 $. By Lemma~\ref{lem31}, we get $ \lim _{k \rightarrow \infty}(1-\phi \frac{\gamma_{k}}{\gamma_{k+1}})=1-\phi>0 $, which means that there exists $k_{0} \in \mathbb{N}$ such that $(1-\phi \frac{\gamma_{k}}{\gamma_{k+1}})>0, \, \forall k \geq k_{0} $. On account of Lemma~\ref{lem31} and Lemma~\ref{lem32}, we deduce that
\begin{equation}\label{eqq}
\|z^{k}-u\| \leq\|s^{k}-u\|,\quad \forall k \geq k_{0}\,.
\end{equation}

\noindent\textbf{Claim 1.} The sequence $\{x^{k}\}$ is bounded.   According to the definition of $ \{x^{k+1}\} $, one has
\begin{equation}\label{eqa}
\begin{aligned}
\|x^{k+1}-u\| &=\|(1-\theta_{k}-\eta_{k})(z^{k}-u)+\eta_{k}(T z^{k}-u)-\theta_{k} u\| \\
& \leq\|(1-\theta_{k}-\eta_{k})(z^{k}-u)+\eta_{k}(T z^{k}-u)\|+\theta_{k}\|u\|\,.
\end{aligned}
\end{equation}
Combining \eqref{eq21}, \eqref{eq23} and \eqref{eqq}, we have
\[
\begin{aligned}
&\quad\|(1-\theta_{k}-\eta_{k})(z^{k}-u)+\eta_{k}(T z^{k}-u)\|^{2} \\
=&(1-\theta_{k}-\eta_{k})^{2}\|z^{k}-u\|^{2}+\eta_{k}^{2}\|T z^{k}-u\|^{2}+2(1-\theta_{k}-\eta_{k}) \eta_{k}\langle T z^{k}-u, z^{k}-u\rangle \\
\leq&(1-\theta_{k}-\eta_{k})^{2}\|z^{k}-u\|^{2}+\eta_{k}^{2}\big[\|z^{k}-u\|^{2}+\lambda\|z^{k}-T z^{k}\|^{2}\big]  \\
&+2(1-\theta_{k}-\eta_{k}) \eta_{k}\big[\|z^{k}-u\|^{2}-\frac{1-\lambda}{2}\|z^{k}-T z^{k}\|^{2}\big] \\
=&(1-\theta_{k})^{2}\|z^{k}-u\|^{2}+\eta_{k}(\eta_{k}-(1-\lambda)(1-\theta_{k}))\|z^{k}-T z^{k}\|^{2} \\
\leq& (1-\theta_{k})^{2}\|s^{k}-u\|^{2}\,,
\end{aligned}
\]
which implies that
\begin{equation}\label{eqz}
\|(1-\theta_{k}-\eta_{k})(z^{k}-u)+\eta_{k}(T z^{k}-u)\| \leq(1-\theta_{k})\|s^{k}-u\|\,.
\end{equation}
From the definition of $ \{s^{k}\} $, we can write
\begin{equation}\label{c}
\begin{aligned}
\|s^{k}-u\| 
&\leq\|x^{k}-u\|+\theta_{k} \cdot \frac{\delta_{k}}{\theta_{k}}\|x^{k}-x^{k-1}\|\,.
\end{aligned}
\end{equation}
By Remark~\ref{rem31}, we get that $\frac{\delta_{k}}{\theta_{k}}\|x^{k}-x^{k-1}\| \rightarrow 0$. Thus, there exists a constant $Q_{1}>0$ such that
\begin{equation}\label{r}
\frac{\delta_{k}}{\theta_{k}}\|x^{k}-x^{k-1}\| \leq Q_{1}, \quad \forall k \geq 1\,.
\end{equation}
From \eqref{eqq}, \eqref{c} and \eqref{r}, we find that
\begin{equation}\label{f}
\|z^{k}-u\| \leq\|s^{k}-u\| \leq\|x^{k}-u\|+\theta_{k} Q_{1},\quad \forall k \geq k_{0}\,.
\end{equation}
Combining \eqref{eqa}, \eqref{eqz} and \eqref{f},  we obtain
\[
\begin{aligned}
\|x^{k+1}-u\| & \leq(1-\theta_{k})\|s^{k}-u\|+\theta_{k}\|u\| \\
&\leq(1-\theta_{k})\|x^{k}-u\|+\theta_{k}(\|u\|+Q_{1})\\
& \leq \max \big\{\|x^{k}-u\|,\|u\|+Q_{1}\big\} \\
& \leq \cdots \leq \max \left\{\|x^{0}-u\|,\|u\|+Q_{1}\right\}\,.
\end{aligned}
\]
Thus, the sequence $\{x^{k}\}$ is bounded. So the sequences $ \{s^{k}\} $ and $\{z^{k}\}$ are also bounded.

\noindent\textbf{Claim 2.}
\[
\begin{aligned}
&\quad\eta_{k}\left[(1-\lambda)-\eta_{k}\right]\|z^{k}-T z^{k}\|^{2}+\Big(1-\phi \frac{\gamma_{k}}{\gamma_{k+1}}\Big)\|y^{k}-s^{k}\|^{2}+\Big(1-\phi \frac{\gamma_{k}}{\gamma_{k+1}}\Big)\|z^{k}-y^{k}\|^{2}\\ &\leq\|x^{k}-u\|^{2}-\|x^{k+1}-u\|^{2}+\theta_{k} Q_{4}\,.
\end{aligned}
\]
Indeed, it follows from \eqref{f} that
\begin{equation}\label{eqo}
\begin{aligned}
\|s^{k}-u\|^{2} & \leq(\|x^{k}-u\|+\theta_{k} Q_{1})^{2} \\
&=\|x^{k}-u\|^{2}+\theta_{k}(2 Q_{1}\|x^{k}-u\|+\theta_{k} Q_{1}^{2}) \\
&  \leq\|x^{k}-u\|^{2}+\theta_{k} Q_{2}
\end{aligned}
\end{equation}
for some $Q_{2}>0 $. Using \eqref{eq22}, \eqref{f},  \eqref{eqo}  and Lemma~\ref{lem32}, we obtain
\begin{equation*}
\begin{aligned}
\|x^{k+1}-u\|^{2} =&\|(z^{k}-u)+\eta_{k}(T z^{k}-z^{k})-\theta_{k} z^{k}\|^{2} \\
\leq&\|(z^{k}-u)+\eta_{k}(T z^{k}-z^{k})\|^{2}-2 \theta_{k}\langle z^{k}, x^{k+1}-u\rangle \\
=&\|z^{k}-u\|^{2}+\eta_{k}^{2}\|T z^{k}-z^{k}\|^{2}+2 \eta_{k}\langle T z^{k}-z^{k}, z^{k}-u\rangle+2 \theta_{k}\langle z^{k}, u-x^{k+1}\rangle\\
\leq&\|z^{k}-u\|^{2}+\eta_{k}^{2}\|T z^{k}-z^{k}\|^{2}+\eta_{k}(\lambda-1)\|z^{k}-T z^{k}\|^{2}+\theta_{k} Q_{3} \\
\leq&\|x^{k}-u\|^{2}+\theta_{k} Q_{4} -\eta_{k}\left[(1-\lambda)-\eta_{k}\right]\|z^{k}-T z^{k}\|^{2}\\
&-\Big(1-\phi \frac{\gamma_{k}}{\gamma_{k+1}}\Big)\|y^{k}-s^{k}\|^{2} -\Big(1-\phi \frac{\gamma_{k}}{\gamma_{k+1}}\Big)\|z^{k}-y^{k}\|^{2}\,,
\end{aligned}
\end{equation*}
where $ Q_{4}=Q_{2}+Q_{3} $. Thus, we can obtain the desired result through a direct calculation.

\noindent\textbf{Claim 3.}
\begin{equation*}
\begin{aligned}
\|x^{k+1}-u\|^{2}\leq &(1-\theta_{k})\|x^{k}-u\|^{2} +\theta_{k}\big[2\eta_{k}\|z^{k}-T z^{k}\|\|x^{k+1}-u\|+2\langle u, u-x^{k+1}\rangle\Big.\\
&\Big.+\frac{3 M \delta_{k}}{ \theta_{k}}\|x^{k}-x^{k-1}\|\big], \quad  \forall k \geq k_{0}\,.
\end{aligned}
\end{equation*}
Indeed, setting $t^{k}=(1-\eta_{k}) z^{k}+\eta_{k} T z^{k}$, and using \eqref{eq21} and \eqref{eq23}, we obtain
\begin{equation}
\begin{aligned}
\|t^{k}-u\|^{2} =&\|(1-\eta_{k})(z^{k}-u)+\eta_{k}(T z^{k}-u)\|^{2} \\
=&(1-\eta_{k})^{2}\|z^{k}-u\|^{2}+\eta_{k}^{2}\|T z^{k}-u\|^{2}+2(1-\eta_{k}) \eta_{k}\langle T z^{k}-u, z^{k}-u\rangle\\ \leq&(1-\eta_{k})^{2}\|z^{k}-u\|^{2}+\eta_{k}^{2}\|z^{k}-u\|^{2}+\eta_{k}^{2} \lambda\|T z^{k}-z^{k}\|^{2} \\
&+2(1-\eta_{k}) \eta_{k}\big[\|z^{k}-u\|^{2}-\frac{1-\lambda}{2}\|T z^{k}-z^{k}\|^{2}\big] \\
=&\|z^{k}-u\|^{2}+\eta_{k}\left[\eta_{k}-(1-\lambda)\right]\|T z^{k}-z^{k}\|^{2}\,.
\end{aligned}
\end{equation}
In view of $ \{\eta_{k}\}\subset(0,1-\lambda) $ and  \eqref{f}, we get
\begin{equation}\label{eqw}
\|t^{k}-u\| \leq\|s^{k}-u\|,\quad \forall k \geq k_{0}\,.
\end{equation}
According to the definition of $ \{s^{k}\} $, one obtains
\begin{equation}\label{eqk}
\begin{aligned}
\|s^{k}-u\|^{2} 
&=\|x^{k}-u\|^{2}+2 \delta_{k}\langle x^{k}-u, x^{k}-x^{k-1}\rangle+\delta_{k}^{2}\|x^{k}-x^{k-1}\|^{2} \\
&\leq\|x^{k}-u\|^{2} + 3M\delta_{k}\|x^{k}-x^{k-1}\|\,,
\end{aligned}
\end{equation}
where $M:=\sup _{n \in \mathbb{N}}\left\{\|x^{k}-u\|, \delta\|x^{k}-x^{k-1}\|\right\}>0$. Moreover, one sees that
\[
\begin{aligned}
x^{k+1} &=t^{k}-\theta_{k} z^{k} \\
&=(1-\theta_{k}) t^{k}-\theta_{k}(z^{k}-t^{k}) \\
&=(1-\theta_{k}) t^{k}-\theta_{k} \eta_{k}(z^{k}-T z^{k})\,.
\end{aligned}
\]
From \eqref{eqw} and \eqref{eqk}, we obtain
\[
\begin{aligned}
\|x^{k+1}-u\|^{2}=&\|(1-\theta_{k})(t^{k}-u)-\theta_{k}( \eta_{k}(z^{k}-T z^{k})+ u)\|^{2} \\
\leq&(1-\theta_{k})^{2}\|t^{k}-u\|^{2}-2\theta_{k}\langle \eta_{k}(z^{k}-T z^{k})+ u, x^{k+1}-u\rangle \\
=&(1-\theta_{k})^{2}\|t^{k}-u\|^{2}+ \theta_{k} \big[2\eta_{k}\langle z^{k}-T z^{k}, u-x^{k+1}\rangle +2\langle u, u-x^{k+1}\rangle\big] \\
\leq &(1-\theta_{k})\|x^{k}-u\|^{2} +\theta_{k}\big[2\eta_{k}\|z^{k}-T z^{k}\|\|x^{k+1}-u\|+2\langle u, u-x^{k+1}\rangle\big.\\
&\big.+\frac{3 M \delta_{k}}{ \theta_{k}}\|x^{k}-x^{k-1}\|\big], \quad  \forall k \geq k_{0}\,.
\end{aligned}
\]
\noindent\textbf{Claim 4.} The sequence $\{\|x^{k}-u\|^{2}\}$ converges to zero. We regard to two reasonable situations on the sequence $\{\|x^{k}-u\|^{2}\}$.

{\it Case 1:} There exists an $N \in \mathbb{N}$ such that $\|x^{k+1}-u\|^{2} \leq\|x^{k}-u\|^{2}$ for all $k \geq N $. This implies that $\lim _{k \rightarrow \infty}\|x^{k}-u\|^{2}$ exists. Since $\lim _{k \rightarrow \infty}\big(1-\phi \frac{\gamma_{k}}{\gamma_{k+1}}\big)=1-\phi>0$ and  Claim 2, we obtain
\begin{equation}\label{v}
\lim _{k \rightarrow \infty}\|s^{k}-y^{k}\|=0, \text{ and }\lim _{k \rightarrow \infty}\|z^{k}-T z^{k}\|=0,\text{ and }
\lim _{k \rightarrow \infty}\|z^{k}-y^{k}\|=0\,,
\end{equation}
which implies that $ \lim _{k \rightarrow \infty}\|z^{k}-s^{k}\|=0 $.
According to the definition of $ \{s^{k}\} $, one has
\begin{equation}\label{eqr}
\|x^{k}-s^{k}\|=\delta_{k}\|x^{k}-x^{k-1}\|=\theta_{k} \cdot \frac{\delta_{k}}{\theta_{k}}\|x^{k}-x^{k-1}\| \rightarrow 0\,.
\end{equation}
This together with $ \lim _{k \rightarrow \infty}\|z^{k}-s^{k}\|=0 $ implies that
\begin{equation}\label{y}
\lim _{k \rightarrow \infty}\|z^{k}-x^{k}\|=0\,.
\end{equation}
Combining condition \ref{con4}, \eqref{v} and \eqref{y}, we have
\[
\|x^{k+1}-x^{k}\| \leq\|z^{k}-x^{k}\|+\theta_{k}\|z^{k}\|+\eta_{k}\|z^{k}-T z^{k}\| \rightarrow 0\,.
\]
We suppose that there exists a subsequence $\{x^{k_{j}}\}$ of $\{x^{k}\}$ such that $x^{k_{j}} \rightharpoonup q$, since $\{x^{k}\}$ is bounded. Hence, we get
\[
\limsup _{k \rightarrow \infty}\langle u, u-x^{k}\rangle=\lim _{j \rightarrow \infty}\langle u, u-x^{k_{j}}\rangle=\langle u, u-q\rangle\,.
\]
One sees that $s^{n_{j}} \rightharpoonup q$ because of \eqref{eqr}, which combining $\lim _{k \rightarrow \infty} \gamma_{k}=\gamma$ and  \eqref{v}, we concluded that $q \in \mathrm{VI}(C, A)$ by means of Lemma~\ref{lem21}. Furthermore,  we get that $z^{n_{j}} \rightharpoonup q$ from \eqref{y}, which combining $\lim _{k \rightarrow \infty}\|z^{k}-T z^{k}\|=0$ implies that $q \in \operatorname{Fix}(T)$. Thus, we have $q \in \mathrm{VI}(C, A) \cap \operatorname { Fix }(T)$. From $u=P_{\mathrm{VI}(C, A) \cap \operatorname { Fix }(T)} 0$, one infers that $ \limsup _{k \rightarrow \infty}\langle u, u-x^{k}\rangle=\langle u, u-q\rangle \leq 0 $. By $\|x^{k+1}-x^{k}\| \rightarrow 0$, we obtain
\[
\limsup _{k \rightarrow \infty}\langle u, u-x^{k+1}\rangle \leq 0\,.
\]
Thus, Combining Claim 3 and Lemma~\ref{lem23}, we deduce that $\lim _{k \rightarrow \infty}\|x^{k+1}-u\|^{2}=0$. This means that $x^{k}\rightarrow u$.

{\it Case 2:} There is a subsequence $\{\|x^{k_{j}}-u\|^{2}\}$ of $\{\|x^{k}-u\|^{2}\}$, which, for all $j \in \mathbb{N} $, satisfies $\|x^{k_{j}}-u\|^{2}<\|x^{k_{j}+1}-u\|^{2}$. In this situation, according to Lemma~\ref{lem22}, there is a nondecreasing sequence $\left\{m_{k}\right\}$ of $\mathbb{N}$ such that $\lim _{k \rightarrow \infty} m_{k}=\infty$, and the following conclusions hold for all $k \in \mathbb{N}$:
\begin{equation}\label{eqs}
\|x^{m_{k}}-u\|^{2} \leq\|x^{m_{k}+1}-u\|^{2}, \quad \text { and } \quad\|x^{k}-u\|^{2} \leq\|x^{m_{k}+1}-u\|^{2}\,.
\end{equation}
From Claim 2, we have
\[
\begin{aligned}
&\quad\Big(1-\phi \frac{\gamma_{m_{k}}}{\gamma_{m_{k}+1}}\Big)\|y^{m_{k}}-s^{m_{k}}\|^{2}+\Big(1-\phi \frac{\gamma_{m_{k}}}{\gamma_{m_{k}+1}}\Big)\|z^{m_{k}}-y^{m_{k}}\|^{2} \\
&\quad+\eta_{m_{k}}\left[(1-\lambda)-\eta_{m_{k}}\right]\|z^{m_{k}}-T z^{m_{k}}\|^{2}\\
&\leq\|x^{m_{k}}-u\|^{2}-\|x^{m_{k}+1}-u\|^{2}+\theta_{m_{k}} Q_{4} \leq \theta_{m_{k}} Q_{4}\,.
\end{aligned}
\]
From condition \ref{con4}, we obtain
\begin{equation*}
\lim _{k \rightarrow \infty}\|s^{m_{k}}-y^{m_{k}}\|=\lim _{k \rightarrow \infty}\|z^{m_{k}}-y^{m_{k}}\|=
\lim _{k \rightarrow \infty}\|z^{m_{k}}-T z^{m_{k}}\|=0.
\end{equation*}
As proved in the first situation, we get that $ \limsup _{k \rightarrow \infty}\langle u, u-x^{m_{k}+1}\rangle \leq 0 $. From Claim 3 and \eqref{eqs}, we obtain
\[
\begin{aligned}
\|x^{m_{k}+1}-u\|^{2}
\leq &(1-\theta_{m_{k}})\|x^{m_{k}+1}-u\|^{2}+\theta_{m_{k}}\big[ 2\eta_{m_{k}}\|z^{m_{k}}-T z^{m_{k}}\|\|x^{m_{k}+1}-u\|\Big.
\\&\Big.+2\langle u, u-x^{m_{k}+1}\rangle +\frac{3 M \delta_{m_{k}}}{ \theta_{m_{k}}}\|x^{m_{k}}-x^{m_{k}-1}\|\big]\,,
\end{aligned}
\]
which implies that
\[
\begin{aligned}
\|x_{k}-u\|^{2} \leq\|x^{m_{k}+1}-u\|^{2} \leq& 2 \eta_{m_{k}}\|z^{m_{k}}-T z^{m_{k}}\|\|x^{m_{k}+1}-u\|+2\langle u, u-x^{m_{k}+1}\rangle\\
&+\frac{3 M \delta_{m_{k}}}{ \theta_{m_{k}}}\|x^{m_{k}}-x^{m_{k}-1}\|\,.
\end{aligned}
\]
Thus, $\lim \sup _{k \rightarrow \infty}\|x^{k}-u\| \leq 0$, that is $x^{k} \rightarrow u $. We have thus proved the theorem. \qed
\end{proof}

\subsection{The inertial Mann-type Tseng's extragradient algorithm}
In this subsection, we will introduce a new iteration scheme combining inertial Tseng's extragradient algorithm and Mann-type method. Note that this method only involves the calculation of one projection in each iteration. Our algorithm is as follows:
\begin{algorithm}[H]
\caption{The inertial Mann-type Tseng's extragradient algorithm}
\label{alg2}
\begin{algorithmic}
\STATE {\textbf{Initialization:} Give  $ \delta>0 $, $\gamma_{1}>0$, $\phi \in(0,1)$. Let $x^{0},x^{1} \in H$ be two arbitrary initial points.}
\STATE \textbf{Iterative Steps}: Calculate the next iteration point $ x^{k+1} $ as follows:
\[
\left\{\begin{aligned}
&s^{k}=x^{k}+\delta_{k}(x^{k}-x^{k-1})\,,\\
&y^{k}=P_{C}(s^{k}-\gamma_{k} A s^{k})\,,\\
&z^{k}=y^{k}-\gamma_{k}(A y^{k}-A s^{k})\,,\\
&x^{k+1}=(1-\theta_{k}-\eta_{k}) z^{k}+\eta_{k} T z^{k}\,,
\end{aligned}\right.
\]
update inertial parameter $ \delta_{k} $ and step size $ \gamma_{k+1} $ through \eqref{alpha} and \eqref{lambda2}, respectively.
\end{algorithmic}
\end{algorithm}

The following lemma is crucial to the proof of the convergence of the algorithm.
\begin{lemma}[\cite{tanarxiv}]\label{lem41}
Suppose that Conditions \ref{con1} and \ref{con3} hold. Let sequence $\{z^{k}\}$ be generated by Algorithm~\ref{alg2}. Then, we have
\begin{equation*}
\|z^{k}-u\|^{2} \leq\|s^{k}-u\|^{2}-\Big(1-\phi^{2} \frac{\gamma_{k}^{2}}{\gamma_{k+1}^{2}}\Big)\|s^{k}-y^{k}\|^{2},\quad \forall u \in \mathrm{VI}(C, A)\,,
\end{equation*}
and
\begin{equation*}
\|z^{k}-y^{k}\| \leq \phi \frac{\gamma_{k}}{\gamma_{k+1}}\|s^{k}-y^{k}\| \,.
\end{equation*}
\end{lemma}

\begin{theorem}\label{thm32}
Suppose that Conditions \ref{con1}--\ref{con4} hold. Then the sequence $\{x^{k}\}$ generated by Algorithm~\ref{alg2} converges to  $u \in \operatorname { Fix }(T) \cap \mathrm{VI}(C, A)$  in norm, where $\|u\|=\min \{\|p\|: p \in \operatorname { Fix }(T) \cap \mathrm{VI}(C, A)\}$.
\end{theorem}
\begin{proof}
By $\lim _{k \rightarrow \infty}\big(1-\phi^{2} \frac{\gamma_{k}^{2}}{\gamma_{k+1}^{2}}\big)=1-\phi^{2}>0$, one concludes that there exists $k_{0} \in \mathbb{N}$ such that
\begin{equation}\label{eqd}
1-\phi^{2} \frac{\gamma_{k}^{2}}{\gamma_{k+1}^{2}}>0,\quad \forall k \geq k_{0}\,.
\end{equation}
Combining Lemma \ref{lem41} and \eqref{eqd}, it follows that
\begin{equation}\label{eqx}
\|z^{k}-u\| \leq\|s^{k}-u\|, \quad \forall k \geq k_{0}\,.
\end{equation}
\noindent\textbf{Claim 1.} The sequence $\{x^{k}\}$ is bounded. Using the same arguments as in the Theorem~\ref{thm31} of Claim 1, we get that  $\{x^{k}\}$ is bounded. So $ \{s^{k}\} $ and $\{z^{k}\}$ are bounded.

\noindent\textbf{Claim 2.}
\[
\begin{aligned}
&\quad\eta_{k}\left[(1-\lambda)-\eta_{k}\right]\|z^{k}-T z^{k}\|^{2}+\Big(1-\phi^{2} \frac{\gamma_{k}^{2}}{\gamma_{k+1}^{2}}\Big)\|y^{k}-s^{k}\|^{2}\\ &\leq\|x^{k}-u\|^{2}-\|x^{k+1}-u\|^{2}+\theta_{k} Q_{4}\,.
\end{aligned}
\]
Indeed, using \eqref{eqo} and \eqref{eqx} and Lemma \ref{lem41}, we obtain
\begin{equation}
\begin{aligned}
\|x^{k+1}-u\|^{2}
\leq&\|z^{k}-u\|^{2}+\eta_{k}^{2}\|T z^{k}-z^{k}\|^{2}+2 \eta_{k}\langle T z^{k}-z^{k}, z^{k}-u\rangle+2 \theta_{k}\langle z^{k}, u-x^{k+1}\rangle\\
\leq&\|z^{k}-u\|^{2}+\eta_{k}^{2}\|T z^{k}-z^{k}\|^{2}+\eta_{k}(\lambda-1)\|z^{k}-T z^{k}\|^{2}+\theta_{k} Q_{3} \\
\leq&\|x^{k}-u\|^{2}+\theta_{k} Q_{4} -\eta_{k}\left[(1-\lambda)-\eta_{k}\right]\|z^{k}-T z^{k}\|^{2}\\
&-\Big(1-\phi^{2} \frac{\gamma_{k}^{2}}{\gamma_{k+1}^{2}}\Big)\|y^{k}-s^{k}\|^{2}\,,
\end{aligned}
\end{equation}
where $ Q_{4}=Q_{2}+Q_{3} $. Thus, we can obtain the desired result through a direct calculation.

\noindent\textbf{Claim 3.}
\begin{equation*}
\begin{aligned}
\|x^{k+1}-u\|^{2}\leq &(1-\theta_{k})\|x^{k}-u\|^{2} +\theta_{k}\big[2\eta_{k}\|z^{k}-T z^{k}\|\|x^{k+1}-u\|+2\langle u, u-x^{k+1}\rangle\Big.\\
&\Big.+\frac{3 M \delta_{k}}{ \theta_{k}}\|x^{k}-x^{k-1}\|\big], \quad  \forall k \geq k_{0}\,.
\end{aligned}
\end{equation*}
The desired result can be obtained by using the same arguments as in the Theorem~\ref{thm31} of Claim~3.

\noindent\textbf{Claim 4.} The sequence $\{\|x^{k}-u\|^{2}\}$ converges to zero. The proof is similar to the Claim 4 in Theorem~\ref{thm31}. We leave it to the reader for confirmation. \qed
\end{proof}

\subsection{The modified inertial Mann-type subgradient extragradient algorithm}
In this subsection, we present two new modified inertial Mann-type extragradient algorithms to solve fixed point problems and variational inequality problems. First of all, we assume that the next proposed Algorithm~\ref{alg3} and Algorithm~\ref{alg4} satisfies Conditions \ref{con1}--\ref{con3} and the following Condition~\ref{con5}.
\begin{enumerate}[label=(C5)]
\item Positive sequence $ \{\zeta_{k}\} $ satisfies $\lim_{k \rightarrow \infty} \frac{\zeta_{k}}{1-\theta_{k}}=0$, where $ \{\theta_{k}\}\subset (0,1) $ satisfies
$\lim _{k \rightarrow \infty} (1-\theta_{k})=0$ and $\sum_{n=1}^{\infty} (1-\theta_{k})=\infty$. Let $\left\{\eta_{k}\right\}$ be a real sequence such that $\eta_{k} \in(a, \frac{(1-\lambda) \theta_{k}}{\lambda+\theta_{k}}) \subset(a, \frac{1-\lambda}{1+\lambda}) \subset(a, 1-\lambda)$ for some $a>0$. \label{con5}
\end{enumerate}
Now, we are in a position to show our algorithm, which reads as follows:
\begin{algorithm}[H]
\caption{The modified inertial Mann-type subgradient extragradient algorithm}
\label{alg3}
\begin{algorithmic}
\STATE {\textbf{Initialization:} Give  $ \delta>0 $, $\gamma_{1}>0$, $\phi \in(0,1)$. Let $x^{0},x^{1} \in H$ be two arbitrary initial points.}
\STATE \textbf{Iterative Steps}: Calculate the next iteration point $ x^{k+1} $ as follows:
\[
\left\{\begin{aligned}
&s^{k}=x^{k}+\delta_{k}(x^{k}-x^{k-1})\,,\\
&y^{k}=P_{C}(s^{k}-\gamma_{k} A s^{k})\,,\\
&H_{k}=\big\{x \in H \mid\langle s^{k}-\gamma_{k} A s^{k}-y^{k}, x-y^{k}\rangle \leq 0\big\}\,,\\		
&z^{k}=P_{H_{k}}(s^{k}-\gamma_{k} A y^{k})\,,\\
&x^{k+1}=(1-\eta_{k})(\theta_{k} z^{k})+\eta_{k} T z^{k}\,,
\end{aligned}\right.
\]
update inertial parameter $ \delta_{k} $ and step size $ \gamma_{k+1} $ through \eqref{alpha} and \eqref{lambda2}, respectively.
\end{algorithmic}
\end{algorithm}
\begin{theorem}\label{thm33}
Suppose that Conditions \ref{con1}--\ref{con3} and \ref{con5} hold. Then the iterative sequence $\{x^{k}\}$ formed  by Algorithm~\ref{alg3} converges to  $u \in \operatorname { Fix }(T) \cap \mathrm{VI}(C, A)$  in norm, where $\|u\|=\min \{\|p\|: p \in \operatorname { Fix }(T) \cap \mathrm{VI}(C, A)\}$.
\end{theorem}
\begin{proof}
\noindent \textbf{Claim 1.} The sequence $\{x^{k}\}$ is bounded. From the definition of $ \{s^{k}\} $, one has
\begin{equation}\label{c1}
\begin{aligned}
\|s^{k}-u\|&=\|x^{k}+\delta_{k}(x^{k}-x^{k-1})-u\| \\
&\leq\|x^{k}-u\|+(1-\theta_{k}) \cdot \frac{\delta_{k}}{1-\theta_{k}}\|x^{k}-x^{k-1}\|\,.
\end{aligned}
\end{equation}
According to condition \ref{con5}, we have $\frac{\delta_{k}}{1-\theta_{k}}\|x^{k}-x^{k-1}\| \rightarrow 0$ as $ k \rightarrow \infty $. Thus, there is a constant $Q_{1}>0$ such that
\begin{equation}\label{r1}
\frac{\delta_{k}}{1-\theta_{k}}\|x^{k}-x^{k-1}\| \leq Q_{1}, \quad \forall k \geq 1\,.
\end{equation}
Combining \eqref{eqq}, \eqref{c1} and \eqref{r1}, we find that
\begin{equation}\label{f1}
\|z^{k}-u\| \leq\|s^{k}-u\| \leq\|x^{k}-u\|+(1-\theta_{k}) Q_{1},\quad \forall k \geq k_{0}\,.
\end{equation}
Furthermore, by the definition of $ \{x^{k+1}\} $, one obtains
\begin{equation}\label{eqf}
\begin{aligned}
\|x^{k+1}-u\|  
&=\|\theta_{k} (1-\eta_{k})(z^{k}-u)+\eta_{k}(T z^{k}-u)-(1-\eta_{k})(1-\theta_{k}) u\| \\
& \leq\|\theta_{k} (1-\eta_{k})(z^{k}-u)+\eta_{k}(T z^{k}-u)\|+(1-\eta_{k})(1-\theta_{k})\|u\|\,.
\end{aligned}
\end{equation}
Since $\eta_{k}<\frac{(1-\lambda) \theta_{k}}{\lambda+\theta_{k}}$, one infers that
\[
\lambda \eta_{k}<(1-\lambda) \theta_{k}-\theta_{k} \eta_{k}<\theta_{k}(1-\lambda) (1-\eta_{k})\,.
\]
From \eqref{eq21} and \eqref{eq23}, we get
\begin{equation}\label{eqe}
\begin{aligned}
&\quad\|\theta_{k} (1-\eta_{k})(z^{k}-u)+\eta_{k}(T z^{k}-u)\|^{2} \\
=&(\theta_{k} (1-\eta_{k}))^{2}\|z^{k}-u\|^{2}+\eta_{k}^{2}\|T z^{k}-u\|^{2}+2\theta_{k} (1-\eta_{k}) \eta_{k}\langle T z^{k}-u, z^{k}-u\rangle \\
\leq&(\theta_{k} (1-\eta_{k}))^{2}\|z^{k}-u\|^{2}+\eta_{k}^{2}\|z^{k}-u\|^{2}+\eta_{k}^{2} \lambda\|T z^{k}-z^{k}\|^{2} \\
&+2\theta_{k} (1-\eta_{k}) \eta_{k}\|z^{k}-u\|^{2}-\theta_{k} (1-\lambda)(1-\eta_{k}) \eta_{k}\|T z^{k}-z^{k}\|^{2} \\
=&(\theta_{k} (1-\eta_{k})+\eta_{k})^{2}\|z^{k}-u\|^{2}+\eta_{k}(\lambda \eta_{k}-\theta_{k}(1-\lambda) (1-\eta_{k}))\|T z^{k}-z^{k}\|^{2} \\
\leq&(\theta_{k} (1-\eta_{k})+\eta_{k})^{2}\|z^{k}-u\|^{2}\,,
\end{aligned}
\end{equation}
which combining with \eqref{f1} further yields that
\begin{equation}\label{eqc}
\begin{aligned}
&\quad\|\theta_{k} (1-\eta_{k})(z^{k}-u)+\eta_{k}(T z^{k}-u)\|\\
& \leq(\theta_{k} (1-\eta_{k})+\eta_{k})\|z^{k}-u\| \\
& \leq(1-(1-\eta_{k})(1-\theta_{k}))\|x^{k}-u\|+(1-\theta_{k}) Q_{1}\,.
\end{aligned}
\end{equation}
Combining \eqref{eqf} and \eqref{eqc}, we have
\[
\begin{aligned}
\|x^{k+1}-u\|  \leq&(1-(1-\eta_{k})(1-\theta_{k}))\|x^{k}-u\|\\
&+(1-\eta_{k})(1-\theta_{k})\big[\|u\| + \frac{Q_{1}}{1-\eta_{k}}\big]\\
\leq& \max \big\{\|x^{k}-u\|,\|u\|+ \frac{Q_{1}}{1-\eta_{k}}\big\} \\
\leq& \cdots \leq \max \big\{\|x^{0}-u\|,\|u\|+ \frac{Q_{1}}{1-\eta_{k}}\big\}\,.
\end{aligned}
\]
Consequently, $\{x^{k}\}$ is bounded. So the sequences $ \{s^{k}\} $ and $\{z^{k}\}$ are bounded.

\noindent \textbf{Claim 2.}
\[
\begin{aligned}
&\quad\eta_{k}(1-\lambda-\eta_{k})\|T z^{k}-z^{k}\|^{2}+\Big(1-\phi \frac{\gamma_{k}}{\gamma_{k+1}}\Big)\|y^{k}-s^{k}\|^{2}+\Big(1-\phi \frac{\gamma_{k}}{\gamma_{k+1}}\Big)\|z^{k}-y^{k}\|^{2}\\ &\leq\|x^{k}-u\|^{2}-\|x^{k+1}-u\|^{2}+(1-\theta_{k}) Q_{4}\,,
\end{aligned}
\]
for some $M>0$. Indeed, it follows from \eqref{f1} that
\begin{equation}\label{eqo1}
\begin{aligned}
\|s^{k}-u\|^{2} & \leq(\|x^{k}-u\|+(1-\theta_{k}) Q_{1})^{2} \\
&=\|x^{k}-u\|^{2}+(1-\theta_{k})(2 Q_{1}\|x^{k}-u\|+(1-\theta_{k}) Q_{1}^{2}) \\
& \leq\|x^{k}-u\|^{2}+(1-\theta_{k}) Q_{2}
\end{aligned}
\end{equation}
for some $Q_{2}>0 $. Using \eqref{eq22}, \eqref{f1}, \eqref{eqo1} and Lemma~\ref{lem32}, we obtain
\begin{equation*}
\begin{aligned}
\|x^{k+1}-u\|^{2} 
=&\|(z^{k}-u)+\eta_{k}(T z^{k}-z^{k})-(1-\eta_{k})(1-\theta_{k}) z^{k}\|^{2} \\
\leq&\|z^{k}-u\|^{2}+\eta_{k}^{2}\|T z^{k}-z^{k}\|^{2}+2 \eta_{k}\langle T z^{k}-z^{k}, z^{k}-u\rangle \\
&-2(1-\eta_{k})(1-\theta_{k})\langle z^{k}, x^{k+1}-u\rangle \\
\leq &\|z^{k}-u\|^{2}+\eta_{k}^{2}\|T z^{k}-z^{k}\|^{2}-\eta_{k}(1-\lambda)\|T z^{k}-z^{k}\|^{2} \\
&+2(1-\eta_{k})(1-\theta_{k})\langle z^{k}, u-x^{k+1}\rangle \\
\leq&\|z^{k}-u\|^{2}-\eta_{k}(1-\lambda-\eta_{k})\|T z^{k}-z^{k}\|^{2}+(1-\theta_{k}) Q_{3} \\
\leq &\|x^{k}-u\|^{2}-\eta_{k}(1-\lambda-\eta_{k})\|T z^{k}-z^{k}\|^{2}+(1-\theta_{k}) Q_{4}\\
&-\Big(1-\phi \frac{\gamma_{k}}{\gamma_{k+1}}\Big)\|y^{k}-s^{k}\|^{2} -\Big(1-\phi \frac{\gamma_{k}}{\gamma_{k+1}}\Big)\|z^{k}-y^{k}\|^{2}\,,
\end{aligned}
\end{equation*}
where $ Q_{4}=Q_{2}+Q_{3} $. Thus, we can obtain the desired result through a direct calculation.

\noindent \textbf{Claim 3.}
\[
\begin{aligned}
\|x^{k+1}-u\|^{2} \leq &\left[1-(1-\eta_{k})(1-\theta_{k})\right]\|x^{k}-u\|^{2}+(1-\eta_{k})(1-\theta_{k})\big[2\langle u, u-x^{k+1}\rangle\big.\\
&\big.+2 \eta_{k}\|T z^{k}-z^{k}\|\|x^{k+1}-u\|+\frac{3M}{(1-\eta_{k})}\cdot\frac{\delta_{k}}{(1-\theta_{k})}\|x^{k}-x^{k-1}\|\big]\,.
\end{aligned}
\]
Indeed, take $t^{k}=(1-\eta_{k}) z^{k}+\eta_{k} T z^{k}$, then as proved in Claim 3 of Theorem \ref{thm31}, we get that $\|t^{k}-u\| \leq\|s^{k}-u\|$.  This together with \eqref{eqk} yields that
\[
\begin{aligned}
\|x^{k+1}-u\|^{2}=&\|t^{k}-u-(1-\eta_{k})(1-\theta_{k}) z^{k}\|^{2} \\
=& \|\left[1-(1-\eta_{k})(1-\theta_{k})\right](t^{k}-u) +(1-\eta_{k})(1-\theta_{k})\big[(t^{k}-z^{k})- u\big] \|^{2} \\
=& \|\left[1-(1-\eta_{k})(1-\theta_{k})\right](t^{k}-u) +(1-\eta_{k})(1-\theta_{k}) \big[\eta_{k}(T z^{k}-z^{k})- u\big] \|^{2} \\
\leq &\left[1-(1-\eta_{k})(1-\theta_{k})\right]^{2}\|t^{k}-u\|^{2} \\
&+2(1-\eta_{k})(1-\theta_{k})\langle\eta_{k}(T z^{k}-z^{k})-u, x^{k+1}-u\rangle \\
\leq &\left[1-(1-\eta_{k})(1-\theta_{k})\right]\|x^{k}-u\|^{2}+(1-\eta_{k})(1-\theta_{k})\big[2\langle u, u-x^{k+1}\rangle\big.\\
&\big.+2 \eta_{k}\|T z^{k}-z^{k}\|\|x^{k+1}-u\|+\frac{3M}{(1-\eta_{k})}\cdot\frac{\delta_{k}}{(1-\theta_{k})}\|x^{k}-x^{k-1}\|\big]\,.
\end{aligned}
\]
\noindent \textbf{Claim 4.} The sequence $\{\|x^{k}-u\|^{2}\}$ converges to zero. The proof of this result is similar to that of Theorem~\ref{thm31}. We leave it to the reader for confirmation. \qed
\end{proof}

\subsection{The modified inertial Mann-type Tseng's extragradient algorithm}
Finally, we introduce a modified inertial Mann-type Tseng's extragradient algorithm. The details of the algorithm are described as follows:
\begin{algorithm}[H]
\caption{The modified inertial Mann-type Tseng's extragradient algorithm}
\label{alg4}
\begin{algorithmic}
\STATE {\textbf{Initialization:} Give  $ \delta>0 $, $\gamma_{1}>0$, $\phi \in(0,1)$. Let $x^{0},x^{1} \in H$ be two arbitrary initial points.}
\STATE \textbf{Iterative Steps}: Calculate the next iteration point $ x^{k+1} $ as follows:
\[
\left\{\begin{aligned}
&s^{k}=x^{k}+\delta_{k}(x^{k}-x^{k-1})\,,\\
&y^{k}=P_{C}(s^{k}-\gamma_{k} A s^{k})\,,\\
&z^{k}=y^{k}-\gamma_{k}(A y^{k}-A s^{k})\,,\\
&x^{k+1}=(1-\eta_{k})(\theta_{k} z^{k})+\eta_{k} T z^{k}\,,
\end{aligned}\right.
\]
update inertial parameter $ \delta_{k} $ and step size $ \gamma_{k+1} $ through \eqref{alpha} and \eqref{lambda2}, respectively.
\end{algorithmic}
\end{algorithm}
\begin{theorem}\label{thm34}
Suppose that Conditions \ref{con1}--\ref{con3} and \ref{con5} hold. Then the iterative sequence $\{x^{k}\}$ formed by Algorithm~\ref{alg4} converges to  $u \in \operatorname { Fix }(T) \cap \mathrm{VI}(C, A)$  in norm, where $\|u\|=\min \{\|p\|: p \in \operatorname { Fix }(T) \cap \mathrm{VI}(C, A)\}$.
\end{theorem}
\begin{proof}
\noindent\textbf{Claim 1.} The sequence $\{x^{k}\}$ is bounded. As proved in Theorem~\ref{thm32}, we also get that $ \|z^{k}-u\| \leq\|s^{k}-u\|, \forall k \geq k_{0}$. Using the same arguments as in   Theorem~\ref{thm33} of Claim 1, one concludes that  $\{x^{k}\}$ is bounded. So $ \{s^{k}\} $ and $\{z^{k}\}$ are bounded.

\noindent \textbf{Claim 2.}
\[
\begin{aligned}
&\quad\eta_{k}(1-\lambda-\eta_{k})\|T z^{k}-z^{k}\|^{2}+\Big(1-\phi^{2} \frac{\gamma_{k}^{2}}{\gamma_{k+1}^{2}}\Big)\|y^{k}-s^{k}\|^{2}\\ &\leq\|x^{k}-u\|^{2}-\|x^{k+1}-u\|^{2}+(1-\theta_{k}) Q_{4}\,,
\end{aligned}
\]
for some $Q_{4}>0$. Indeed, using \eqref{eqo1} and Lemma \ref{lem41}, we have
\begin{equation*}
\begin{aligned}
\|x^{k+1}-u\|^{2}
\leq &\|z^{k}-u\|^{2}+\eta_{k}^{2}\|T z^{k}-z^{k}\|^{2}-\eta_{k}(1-\lambda)\|T z^{k}-z^{k}\|^{2} \\
&+2(1-\eta_{k})(1-\theta_{k})\langle z^{k}, u-x^{k+1}\rangle \\
\leq&\|z^{k}-u\|^{2}-\eta_{k}(1-\lambda-\eta_{k})\|T z^{k}-z^{k}\|^{2}+(1-\theta_{k}) Q_{3} \\
\leq &\|x^{k}-u\|^{2}-\eta_{k}(1-\lambda-\eta_{k})\|T z^{k}-z^{k}\|^{2}+(1-\theta_{k}) Q_{4}\\
&-\Big(1-\phi^{2} \frac{\gamma_{k}^{2}}{\gamma_{k+1}^{2}}\Big)\|y^{k}-s^{k}\|^{2} \,,
\end{aligned}
\end{equation*}
where $ Q_{4}=Q_{2}+Q_{3} $. Thus, we can obtain the desired result through a direct calculation.

\noindent\textbf{Claim 3.}
\[
\begin{aligned}
\|x^{k+1}-u\|^{2} \leq &\left[1-(1-\eta_{k})(1-\theta_{k})\right]\|x^{k}-u\|^{2}+(1-\eta_{k})(1-\theta_{k})\big[2\langle u, u-x^{k+1}\rangle\big.\\
&\big.+2 \eta_{k}\|T z^{k}-z^{k}\|\|x^{k+1}-u\|+\frac{3M}{(1-\eta_{k})}\cdot\frac{\delta_{k}}{(1-\theta_{k})}\|x^{k}-x^{k-1}\|\big]\,.
\end{aligned}
\]
The desired result can be obtained by using the same arguments as in the Theorem~\ref{thm33} of Claim~3.

\noindent\textbf{Claim 4.} The sequence $\{\|x^{k}-u\|^{2}\}$ converges to zero. The proof is similar to Claim 4 in Theorem~\ref{thm31}. We leave it to the reader for confirmation. \qed
\end{proof}
\section{Numerical examples}\label{sec4}
In this section, we provide several computational tests to illustrate the numerical behavior of our proposed algorithms (For convenience, we abbreviate, Algorithm~\ref{alg1} (iMSEGM), Algorithm~\ref{alg2} (iMTEGM), Algorithm~\ref{alg3} (iMMSEGM) and Algorithm~\ref{alg4} (iMMTEGM)) and compare them with some existing strong convergent methods, including  the Halpern subgradient extragradient method \eqref{HSEGM} \cite{KS}, the self adaptive Tseng's extragradient method \eqref{STEGM} \cite{VTEGM}, the Mann-type subgradient extragradient method \eqref{MSEGM} \cite{THSEGM}, the modified Mann-type subgradient extragradient method \eqref{MMSEGM} \cite{THSEGM}, the Viscosity-type subgradient extragradient method \eqref{VSEGM} \cite{TVNA}, and the  Viscosity-type Tseng's extragradient method \eqref{VTEGM} \cite{TVNA}.

The parameter settings of all algorithms are as follows, see Table~\ref{algset} for details. In our experiment examples, the solution $ x^{*} $ of the problems are known. Therefore, we take $ D_{k}=\|x^{k}-x^{*}\| $ to evaluate the $ k $-th iteration error. Note that the sequence $\left\{D_{k}\right\}\rightarrow 0$ implies that $\{x^{k}\}$ converges to the solution of the problem. In addition, we use the FOM Solver~\cite{FOM} to effectively calculate the projections onto $ C $ and $ H_{k} $. All the programs were implemented in MATLAB 2018a on a Intel(R) Core(TM) i5-8250U CPU @ 1.60GHz computer with RAM 8.00 GB.
\begin{table}[h]
\centering
\caption{Parameter setting for all algorithms}
\label{algset}
\begin{tabular}{rl}
\toprule
Algorithms & Parameters	\\
\midrule
HSEGM &  $ \theta_{k}=1/(k+1) $, $ \eta_{k}=k/(2k+1) $, $ \gamma=0.99/L $.\\	
MSEGM &  $ \theta_{k}=1/(k+1) $, $ \eta_{k}=0.5(1-\theta_{k}) $, $ \gamma=0.99/L $.\\
MMSEGM &  $ \theta_{k}=n/(n+1) $, $ \eta_{k}=\theta_{k}/3 $, $ \gamma=0.99/L $.\\
iMSEGM & $ \theta_{k}=1/(k+1) $, $ \eta_{k}=0.5(1-\theta_{k}) $, $ \delta=0.6 $, $ \zeta_{k}=1/(k+1)^2 $, $ \phi=0.5 $, $ \gamma_{1}=0.5 $.\\
iMTEGM & The parameters set are the same as algorithm (iMSEGM).\\
iMMSEGM & $ \theta_{k}=k/(k+1) $, $ \eta_{k}=\theta_{k}/3 $, $ \delta=0.6 $, $ \zeta_{k}=1/(k+1)^2 $, $ \phi=0.5 $, $ \gamma_{1}=0.5 $.\\
iMTSEGM & The parameters set are the same as algorithm (iMMSEGM).\\
VSEGM & $ \theta_{k}=1/(k+1) $, $ \eta_{k}=k/(2k+1) $, $ \phi=0.5 $, $ \gamma_{1}=0.5 $, $ f(x)=0.5x $.\\
VTEGM & The parameters set are the same as algorithm (VSEGM).\\
STEGM & $ \theta_{k}=1/(k+1) $, $ \eta_{k}=k/(2k+1) $, $ \rho=1 $, $ l=0.5 $, $ \phi=0.4 $, $ \lambda=0.5 $.\\
\bottomrule
\end{tabular}
\end{table}

\begin{example}\label{ex1}
Consider the form of linear operator $ A: R^{n}\rightarrow R^{n} $ ($ n=100, 200 $) as follows: $A(x)=Gx+f$, where $f\in R^n$ and $G=BB^{\mathsf{T}}+S+E$, matrix $B\in R^{n\times n}$, matrix $S\in R^{n\times n}$ is skew-symmetric, and matrix $E\in R^{n\times n}$ is diagonal matrix whose diagonal terms are non-negative (hence $ G $ is positive symmetric definite). We choose the feasible set as $C=\left\{x \in {R}^{n}:-2 \leq x_{i} \leq 5, \, i=1, \ldots, n\right\}$.  It is easy to see that $A$ is Lipschitz continuous monotone and its Lipschitz constant $ L=  \|G\| $.  In this numerical example, both $B, E$ entries are randomly created in $[0,2]$, $S$ is generated randomly in $[-2,2]$ and $ f = 0 $. Let $T: H \rightarrow H$ and $F: H \rightarrow H$ be provided by $T x=0.5 x$ and $F x=0.5 x$, respectively.  We obtain the solution of the problem is $ x^{*}=\{\mathbf{0}\} $. The maximum iteration $ 400 $ as a common stopping criterion and the initial values $ x^{0} = x^{1} $ are randomly generated by \emph{rand(2,1)} in MATLAB. The numerical results with iteration step and elapsed time are shown in Figs.~\ref{fig100_iter}--\ref{fig200_time}.
\begin{figure}[htbp]
\centering
\includegraphics[scale=0.6]{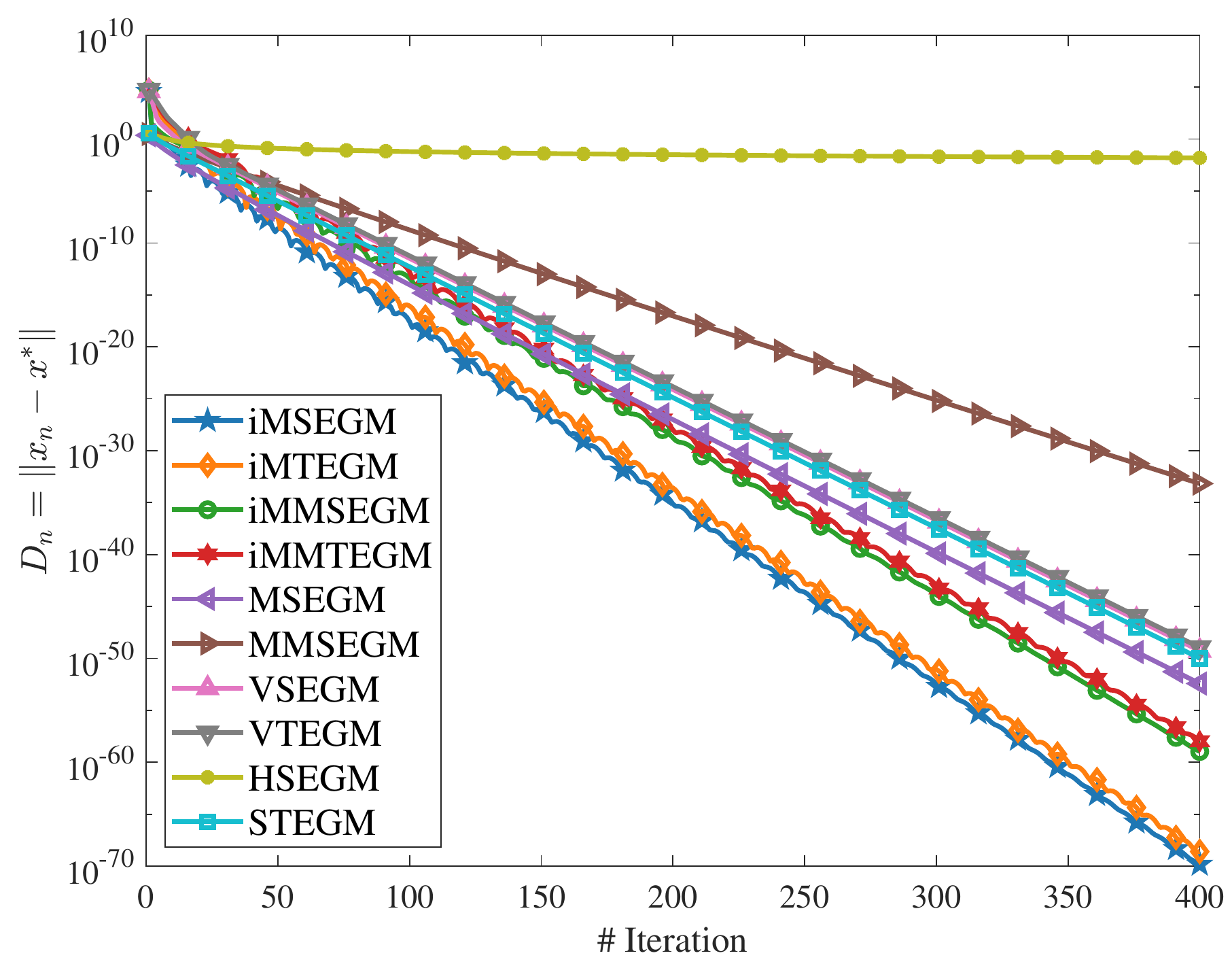}
\caption{Numerical results of Example~\ref{ex1} when $ n=100 $}
\label{fig100_iter}
\end{figure}
\begin{figure}[htbp]
\centering
\includegraphics[scale=0.6]{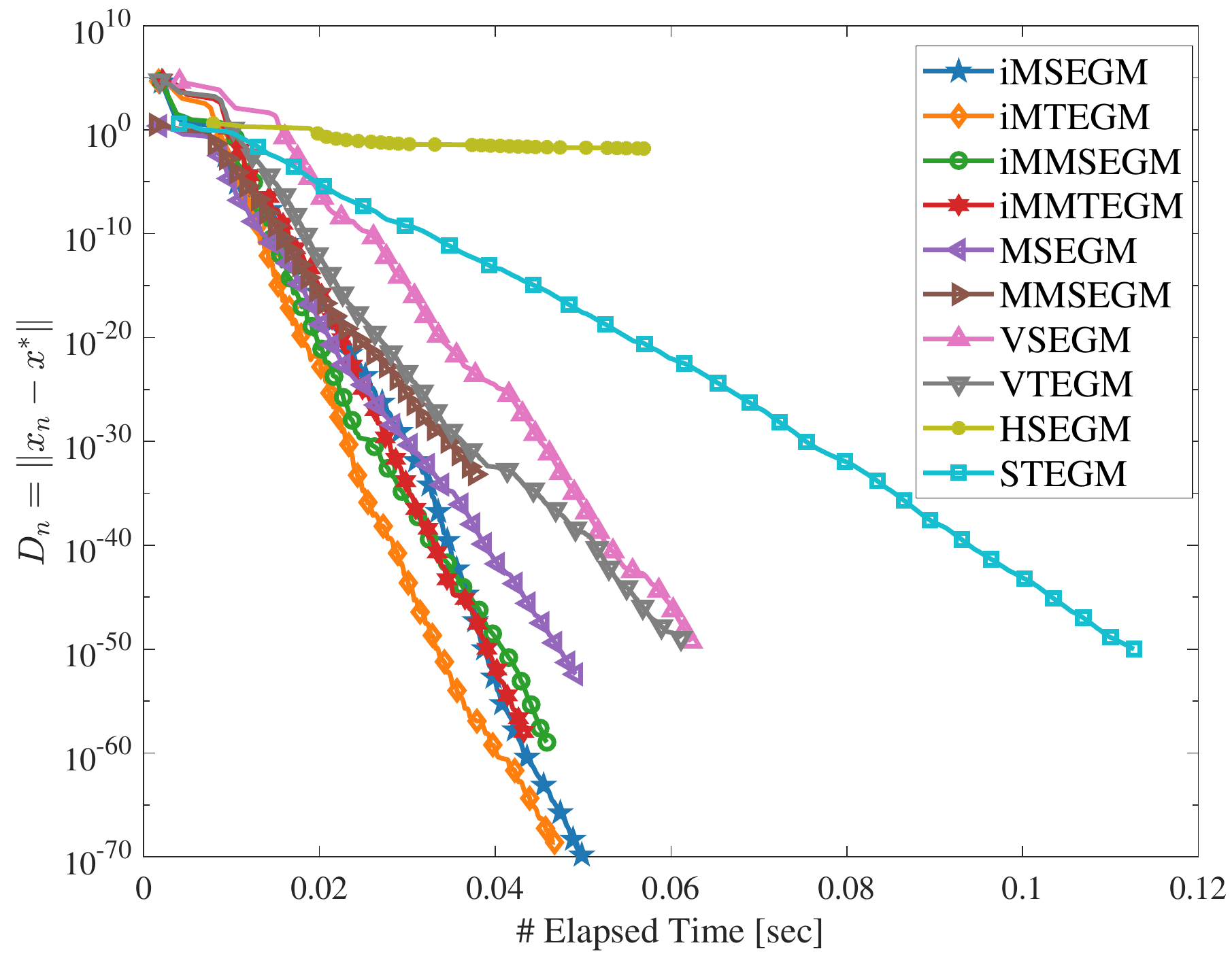}
\caption{Numerical results of Example~\ref{ex1} when $ n=100 $}
\label{fig100_time}
\end{figure}
\begin{figure}[htbp]
\centering
\includegraphics[scale=0.6]{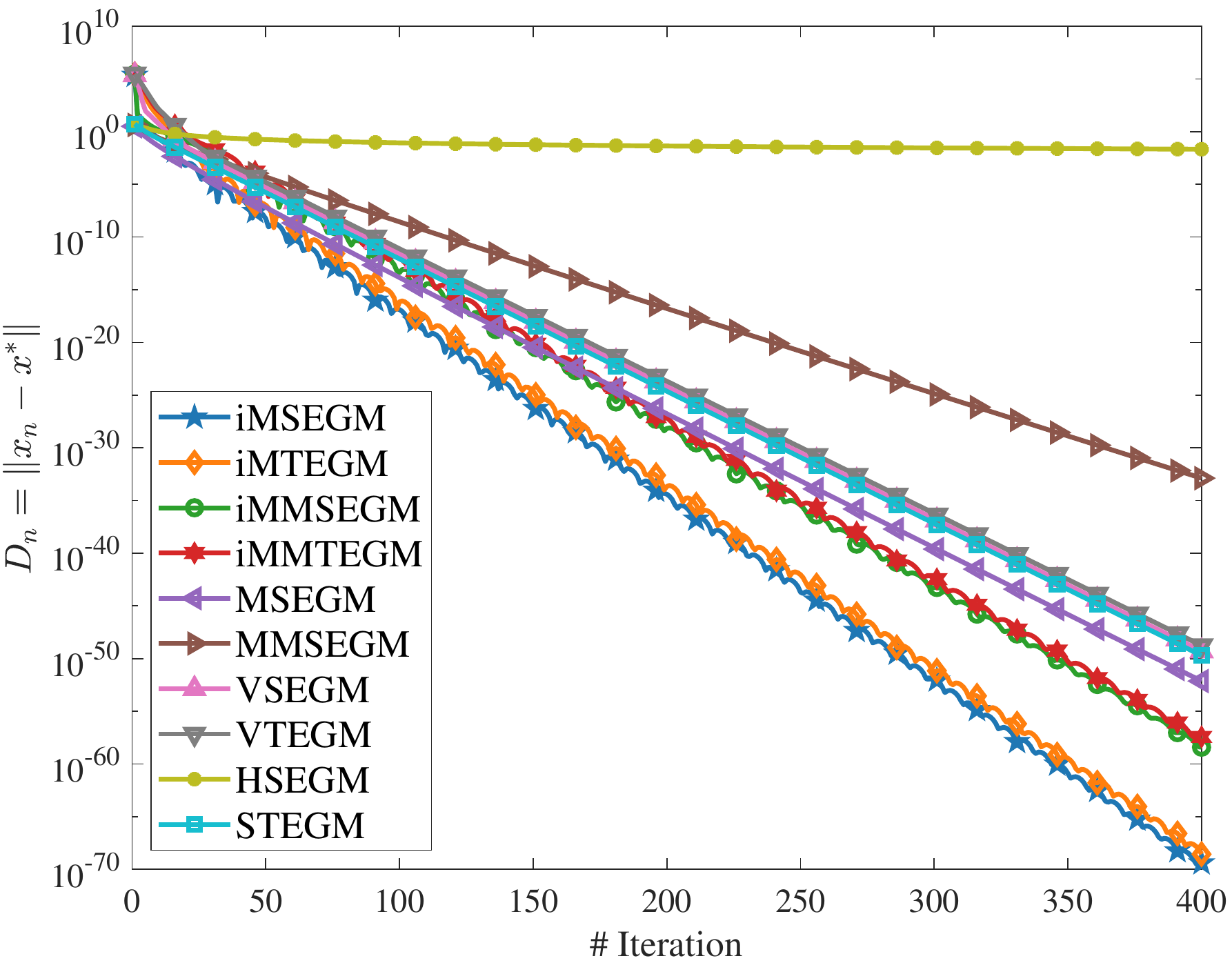}
\caption{Numerical results of Example~\ref{ex1} when $ n=200 $}
\label{fig200_iter}
\end{figure}
\begin{figure}[htbp]
\centering
\includegraphics[scale=0.6]{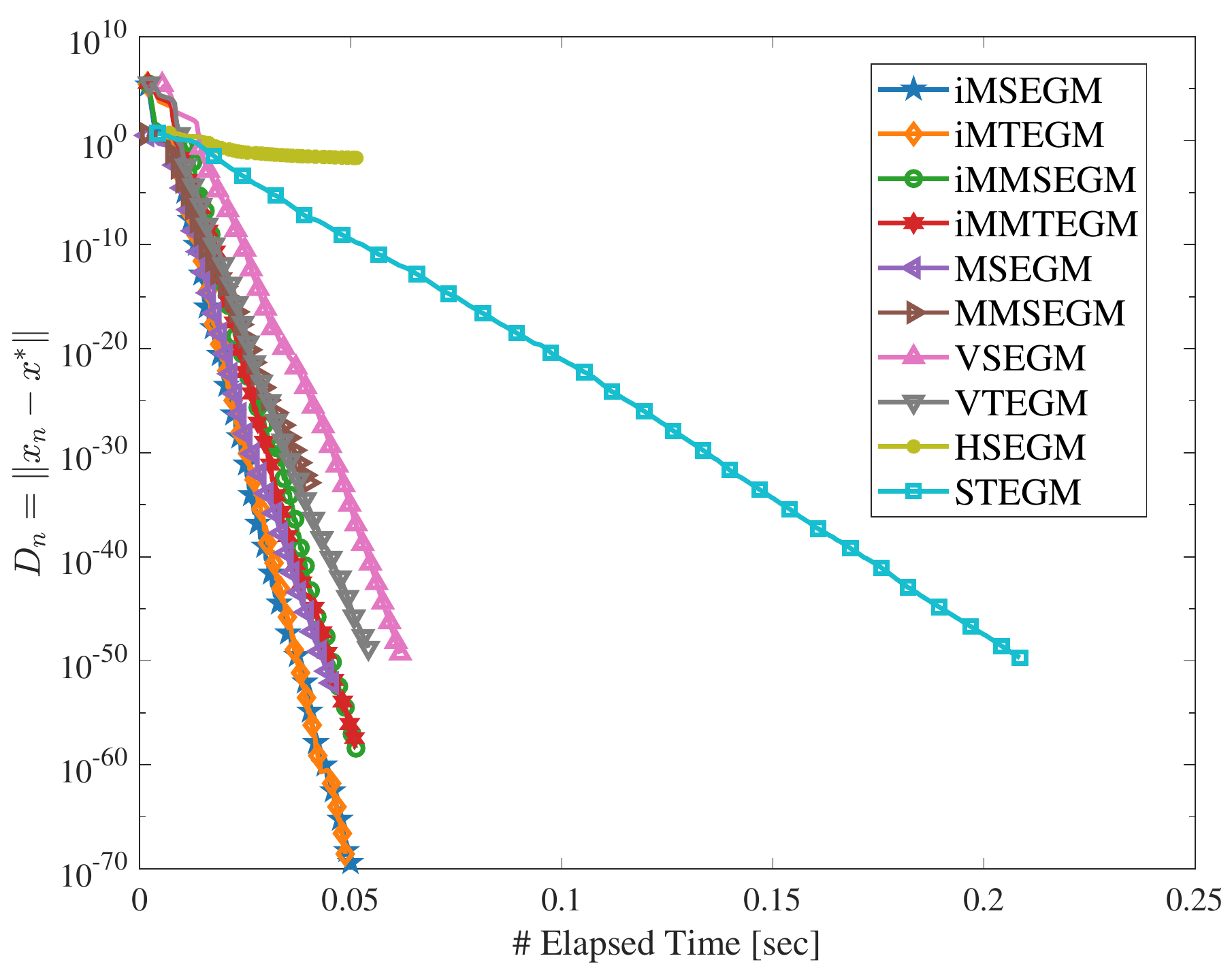}
\caption{Numerical results of Example~\ref{ex1} when $ n=200 $}
\label{fig200_time}
\end{figure}
\end{example}

\begin{example}\label{ex2}
In this numerical example, we focus on a case in Hilbert space $H=L^{2}([0,1])$. Its inner product and induced norm are defined as $\langle m, n\rangle:=\int_{0}^{1} m(t) n(t) \mathrm{d} t$ and $\|m\|:=(\int_{0}^{1}|m(t)|^{2} \mathrm{d} t)^{1 / 2}$, respectively. We choose the feasible set as the unit ball $C:=\{x \in H \mid \|x\| \leq 1\}$. Let operator $A: C \rightarrow H$ be defined as follows:
\[
(A x)(t)=\max \{ x(t),0\}=\frac{x(t)+|x(t)|}{2}\,.
\]
It is easy to verify that $ A $ is monotone and $ 1 $-Lipschitz continuous, and the projection on $ C $ is inherently explicit, that is,
\[
P_{C}(x)=\left\{\begin{array}{ll}
x, & \text { if }\|x\|_{L^{2}} \leq 1; \\
\frac{x}{\|x\|_{L^{2}}}, & \text { if }\|x\|_{L^{2}}>1.
\end{array}\right.
\]
The mapping $T: L^{2}([0,1]) \rightarrow L^{2}([0,1])$ is of the form,
\[
(T x)(t)=\int_{0}^{1} t x(r) \mathrm{d}r, \quad t \in[0,1]\,.
\]
A simple computation indicates that $T$ is $ 0 $-demicontractive and demiclosed at $0 $. Let operator $F: H \rightarrow H$ be  defined as $(F x)(t)=0.5 x(t)$. It is easy to  check that operator $ F $ is  Lipschitz continuous and strongly monotone. Through a straightforward calculation, we know that the solution of the problem is $x^{*}(t)=0 $. The maximum iteration $ 50 $ as a common stopping criterion. With two types of initial points $x^{0}(t)=x^{1}(t)=t^2$ and  $x^{0}(t)=x^{1}(t)=t+0.5\cos(t)$.  The numerical behaviors of $D_{k}=\|x^{k}(t)-x^{*}(t)\|$ with iteration step and elapsed time are described in Figs.~\ref{figt2_time}--\ref{figcost_time}.
\begin{figure}[htbp]
\centering
\includegraphics[scale=0.6]{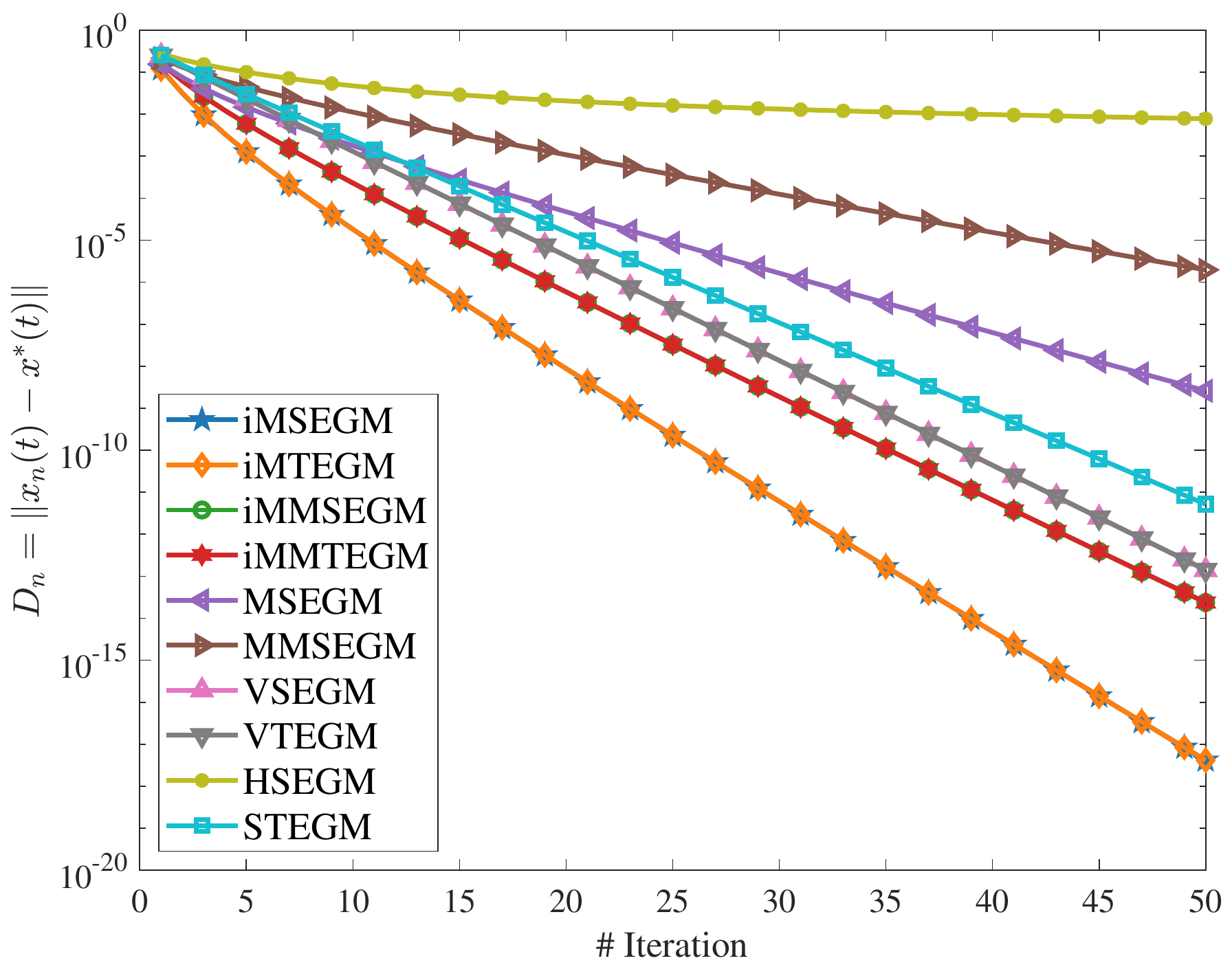}
\caption{Numerical results of Example~\ref{ex2} when $ x^{0}(t)=x^{1}(t)=t^2 $}
\label{figt2_iter}
\end{figure}
\begin{figure}[htbp]
\centering
\includegraphics[scale=0.6]{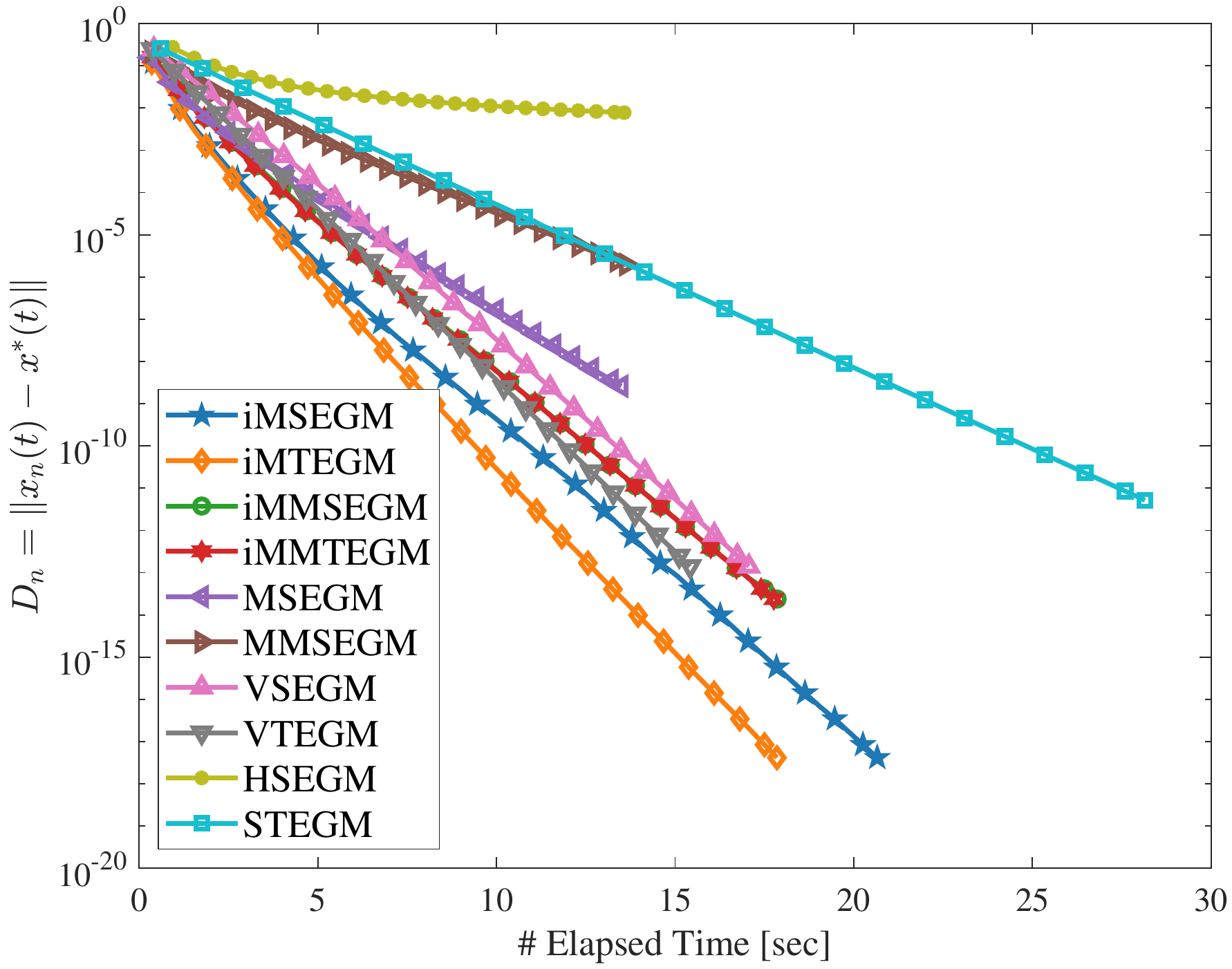}
\caption{Numerical results of Example~\ref{ex2} when $ x^{0}(t)=x^{1}(t)=t^2 $}
\label{figt2_time}
\end{figure}
\begin{figure}[htbp]
\centering
\includegraphics[scale=0.6]{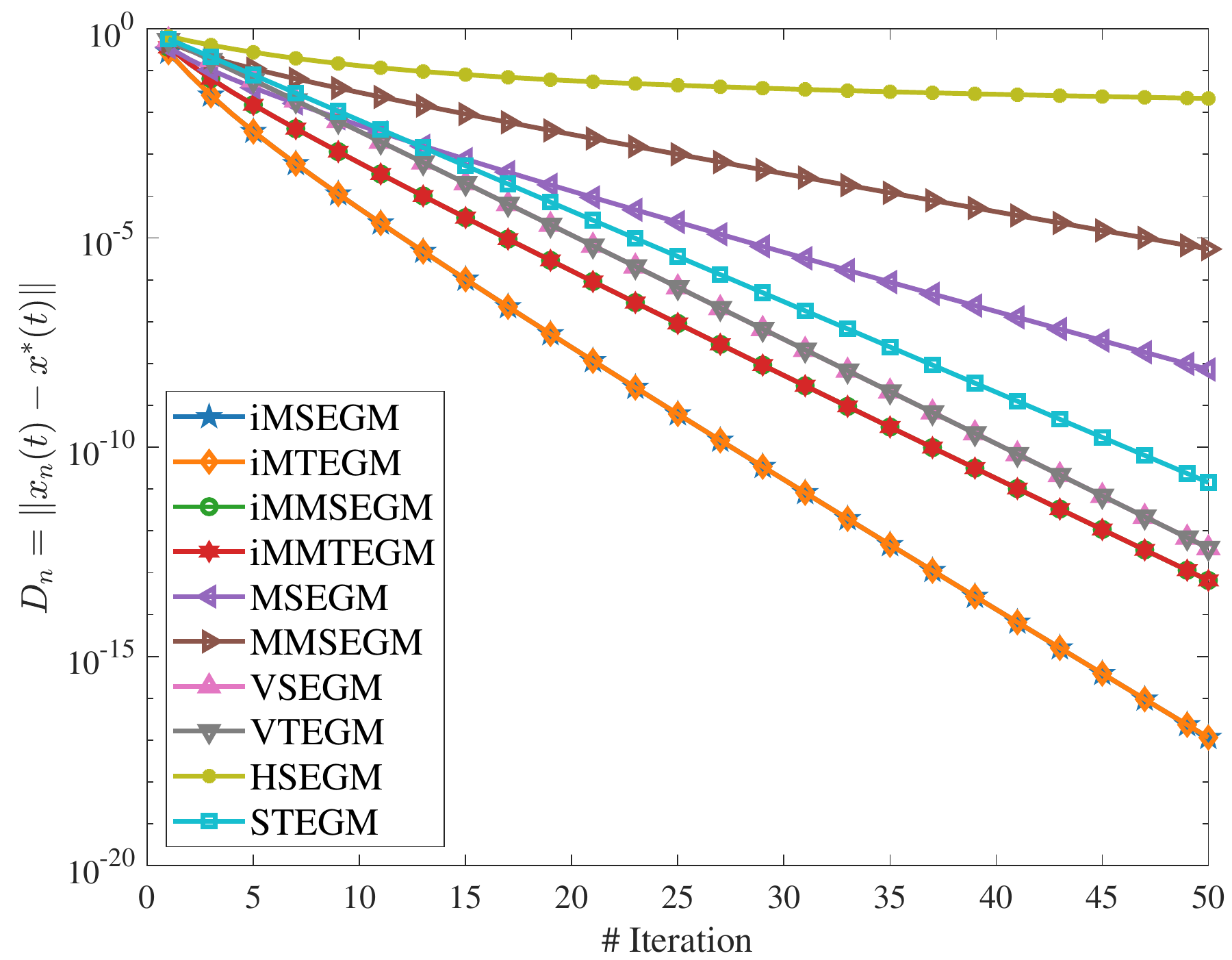}
\caption{Numerical results of Example~\ref{ex2} when $ x^{0}(t)=x^{1}(t)=t+0.5\cos(t)$}
\label{figcost_iter}
\end{figure}
\begin{figure}[htbp]
\centering
\includegraphics[scale=0.6]{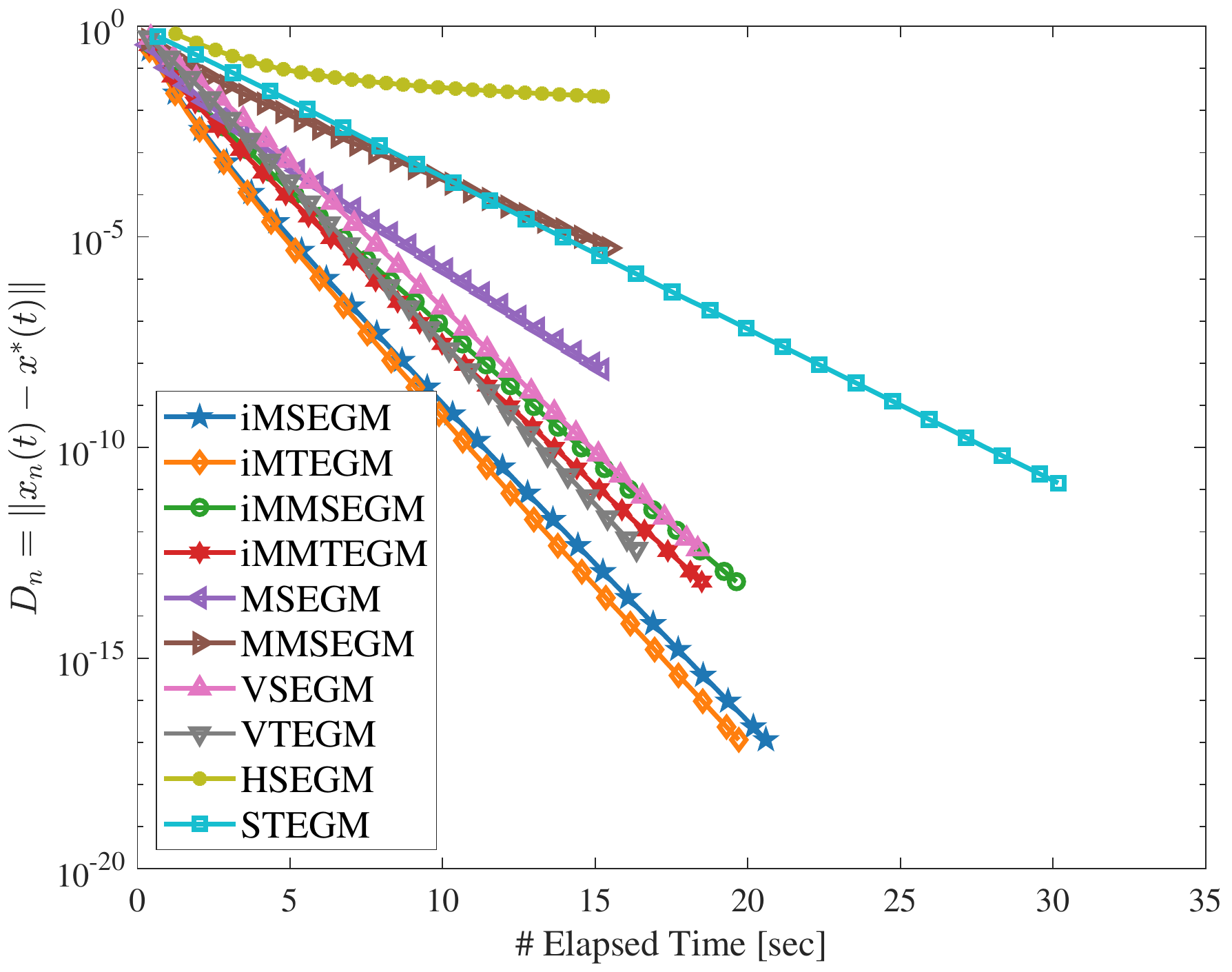}
\caption{Numerical results of Example~\ref{ex2} when $ x^{0}(t)=x^{1}(t)=t+0.5\cos(t) $}
\label{figcost_time}
\end{figure}
\end{example}

\begin{remark}
From the above numerical examples appearing in finite and infinite dimensions, it can be seen that the proposed algorithms have higher convergence accuracy under the same conditions. The convergence speed of our algorithms is faster than that of some known algorithms in the literature, and these results are independent of the size of dimensions and the selection of initial values. More importantly, the algorithms obtained in this paper automatically updates the step size through a simple calculation, which makes our suggested algorithms work without the prior information of the Lipschitz constant of the mapping.
\end{remark}

\section{Final remarks}\label{sec5}
In this research, we presented four new inertial extragradient algorithms with a new simple step size for seeking a common solution of the monotone variational inequality problems and the fixed point problems in a Hilbert space. The advantage of the four algorithms proposed in this paper is that we do not need to know the prior information of Lipschitz constants in advance. In addition, our algorithms add an inertial term, which significantly improves the convergence speed of our algorithms. We have proved strong convergence of the suggested algorithms under certain suitable conditions imposed on parameters. Some numerical examples of finite and infinite dimensions have been presented to demonstrate the performance of the algorithms and compare them with some previously existing ones. The four algorithms obtained in this paper improve and extend the results of some existing literature.

\end{document}